\documentclass[11pt,letterpaper]{amsart}
\usepackage[usenames,dvipsnames]{color}
\usepackage{amsthm,amsfonts,amssymb,amsmath,amsxtra}
\usepackage[all]{xy}
\SelectTips{cm}{}
\usepackage{xr-hyper}
\usepackage[pdfpagelabels,pdftex,hidelinks]{hyperref}
\hypersetup{
  colorlinks   = true, %Colours links
  urlcolor     = blue, %Colour for external hyperlinks
  linkcolor    = blue, %Colour of internal links
  citecolor   = blue, %Colour of citations
pdftitle={},
  pdfauthor={},
  pdfsubject={},
  pdfkeywords={},
  breaklinks=true,
  bookmarksopen=true,
  bookmarksnumbered=true,
  pdfpagemode=UseOutlines,
  plainpages=false}
\usepackage{cleveref}

\usepackage{verbatim}
\usepackage{caption}

\usepackage{lineno}
\usepackage{mathrsfs}

\RequirePackage{xspace}
\RequirePackage{etoolbox}
\RequirePackage{varwidth}
\RequirePackage{enumitem}
\RequirePackage{mathtools}
\RequirePackage{longtable}
\RequirePackage{multirow}

\def\ge{\geqslant}
\def\le{\leqslant}
\def\a{\alpha}
\def\b{\beta}
\def\g{\gamma}

\def\d{\delta}
\def\D{\Delta}

\def\e{\epsilon}

\def\p{\pi}

\def\ps{\psi}

\def\t{\tau}
\def\th{\theta}
\def\k{\kappa}
\def\l{\lambda}

\def\i{^{-1}}

\def\<{\langle}
\def\>{\rangle}

\newcommand{\fkX}{\ensuremath{\mathfrak{X}}\xspace}

\newcommand{\BA}{\ensuremath{\mathbb {A}}\xspace}

\newcommand{\BF}{\ensuremath{\mathbb {F}}\xspace}
\newcommand{{\BG}}{\ensuremath{\mathbb {G}}\xspace}

\newcommand{{\BK}}{\ensuremath{\mathbb {K}}\xspace}

\newcommand{\BQ}{\ensuremath{\mathbb {Q}}\xspace}
\newcommand{\BR}{\ensuremath{\mathbb {R}}\xspace}

\newcommand{\BZ}{\ensuremath{\mathbb {Z}}\xspace}

\newcommand{\CA}{\ensuremath{\mathcal {A}}\xspace}
\newcommand{\CB}{\ensuremath{\mathcal {B}}\xspace}

\newcommand{\CG}{\ensuremath{\mathcal {G}}\xspace}

\newcommand{\CI}{\ensuremath{\mathcal {I}}\xspace}

\newcommand{\CO}{\ensuremath{\mathcal {O}}\xspace}

\newcommand{\CR}{\ensuremath{\mathcal {R}}\xspace}

\newcommand{\CT}{\ensuremath{\mathcal {T}}\xspace}

\newcommand{\GL}{\mathrm{GL}}

\DeclareMathOperator{\Hom}{Hom}

\let\Im\relax
\DeclareMathOperator{\Im}{Im}

\newcommand{\red}{\ensuremath{\mathrm{red}}\xspace}

\newcommand{\ov}{\overline}

\def\tW{\tilde W}

\def\brk{{\breve k}}

\def\COk{{\CO_{\brk}}}

\def\pr{{\rm pr}}
\def\tPhi{\widetilde \Phi}

\def\ind{{\rm ind}}

\def\bx{{\mathbf x}}

\def\der{{\rm der}}
\def\ov{\overline}

\def\tSigma{{\tilde \Sigma}}

\def\re{{\rm re}}
\def\dz{{\dot z}}

\newtheorem{theorem}{Theorem}
\newtheorem{proposition}[theorem]{Proposition}
\newtheorem{lemma}[theorem]{Lemma}

\newtheorem{corollary}[theorem]{Corollary}

\theoremstyle{definition}

\newtheorem{remark}[theorem]{Remark}

\numberwithin{equation}{section}
\numberwithin{theorem}{section}

\setitemize[0]{leftmargin=*,itemsep=\the\smallskipamount}
\setenumerate[0]{leftmargin=*,itemsep=\the\smallskipamount}

\renewcommand{\to}{%
   \ifbool{@display}{\longrightarrow}{\rightarrow}%
   }
\let\shortmapsto\mapsto
\renewcommand{\mapsto}{%
   \ifbool{@display}{\longmapsto}{\shortmapsto}%
   }
\newlength{\olen}
\newlength{\ulen}
\newlength{\xlen}
\newcommand{\xra}[2][]{%
   \ifbool{@display}%
      {\settowidth{\olen}{$\overset{#2}{\longrightarrow}$}%
       \settowidth{\ulen}{$\underset{#1}{\longrightarrow}$}%
       \settowidth{\xlen}{$\xrightarrow[#1]{#2}$}%
       \ifdimgreater{\olen}{\xlen}%
          {\underset{#1}{\overset{#2}{\longrightarrow}}}%
          {\ifdimgreater{\ulen}{\xlen}%
             {\underset{#1}{\overset{#2}{\longrightarrow}}}
             {\xrightarrow[#1]{#2}}}}%
      {\xrightarrow[#1]{#2}}
   }
\makeatother
\newcommand{\xyra}[2][]{%
   \settowidth{\xlen}{$\xrightarrow[#1]{#2}$}%
   \ifbool{@display}%
      {\settowidth{\olen}{$\overset{#2}{\longrightarrow}$}%
       \settowidth{\ulen}{$\underset{#1}{\longrightarrow}$}%
       \ifdimgreater{\olen}{\xlen}%
          {\mathrel{\xymatrix@M=.12ex@C=3.2ex{\ar[r]^-{#2}_-{#1} &}}}%
          {\ifdimgreater{\ulen}{\xlen}%
             {\mathrel{\xymatrix@M=.12ex@C=3.2ex{\ar[r]^-{#2}_-{#1} &}}}
             {\mathrel{\xymatrix@M=.12ex@C=\the\xlen{\ar[r]^-{#2}_-{#1} &}}}}}%
      {\mathrel{\xymatrix@M=.12ex@C=\the\xlen{\ar[r]^-{#2}_-{#1} &}}}%
   }
\makeatletter
\newcommand{\xla}[2][]{%
   \ifbool{@display}%
      {\settowidth{\olen}{$\overset{#2}{\longleftarrow}$}%
       \settowidth{\ulen}{$\underset{#1}{\longleftarrow}$}%
       \settowidth{\xlen}{$\xleftarrow[#1]{#2}$}%
       \ifdimgreater{\olen}{\xlen}%
          {\underset{#1}{\overset{#2}{\longleftarrow}}}%
          {\ifdimgreater{\ulen}{\xlen}%
             {\underset{#1}{\overset{#2}{\longleftarrow}}}
             {\xleftarrow[#1]{#2}}}}%
      {\xleftarrow[#1]{#2}}
   }
\newcommand{\isoarrow}{%
   \ifbool{@display}{\overset{\sim}{\longrightarrow}}{\xrightarrow\sim}%
   }

\newcommand{\sm}{{\,\smallsetminus\,}}

\usepackage{upgreek}
\makeatletter
\newcommand{\colim@}[2]{%
  \vtop{\m@th\ialign{##\cr
    \hfil$#1\operator@font lim$\hfil\cr
    \noalign{\nointerlineskip\kern1.5\ex@}#2\cr
    \noalign{\nointerlineskip\kern-\ex@}\cr}}%
}
\newcommand{\colim}{%
  \mathop{\mathpalette\colim@{\rightarrowfill@\textstyle}}\nmlimits@
}
\newcommand{\prolim@}[2]{%
  \vtop{\m@th\ialign{##\cr
    \hfil$#1\operator@font lim$\hfil\cr
    \noalign{\nointerlineskip\kern1.5\ex@}#2\cr
    \noalign{\nointerlineskip\kern-\ex@}\cr}}%
}
\newcommand{\prolim}{%
  \mathop{\mathpalette\colim@{\leftarrowfill@\textstyle}}\nmlimits@
}
\makeatother

\DeclareMathOperator{\Perf}{Perf}

\begin{document}

\title[]{A geometric construction of supercuspidal representations} 

\author[Ben Liu]{Ben Liu}
\address{Academy of Mathematics and Systems Science, Chinese Academy of Sciences, Beijing 100190, China}
\email{liubenmath@gmail.com}

\author[Sian Nie]{Sian Nie}
\address{Academy of Mathematics and Systems Science, Chinese Academy of Sciences, Beijing 100190, China}

\address{School of Mathematical Sciences, University of Chinese Academy of Sciences, Chinese Academy of Sciences, Beijing 100049, China}
\email{niesian@amss.ac.cn}

\author{Zhihang Yu}
\address{The Second Affiliated Hospital of Chongqing Medical University, Chongqing 400010, China}
\address{The National Center for Applied Mathematics in Chongqing, Chongqing 401331, China}
\email{yuzhihang@amss.ac.cn}

\begin{abstract}
We extend the framework of deep level Deligne-Lusztig varieties to the case of arbitrary residue characteristics $p > 0$. As an application, we provide a construction of a large class of (new) supercuspidal representations for unramified $p$-adic groups with arbitrary residue fields, in a purely geometric way.  
\end{abstract}

\maketitle

\section{Introduction}
Supercuspidal representations are building blocks in the representation theory of $p$-adic groups. A well-known folklore conjecture states that each supercuspidal representation is compactly induced from a representation of some compact modulo center subgroup. For inner forms of $\GL_n$, the conjecture is proved in \cite{BK93} and \cite{SS08}. For classical groups, it is proved in \cite{Ste08} for $p \neq 2$. Based on the influential construction of Yu, it is proved in \cite{Kim} and \cite{Fin21b} for an arbitrary $p$-adic group, provided that $p$ satisfies a mild largeness condition. For small $p$, this folklore conjecture remains wildly open.

It is desirable to provide a construction of supercuspidal representations by compact inductions, which works for all primes $p$. Such an attempt has been carried out in \cite{FS} by extending the framework of Yu's algebraic construction. In this paper, we present a purely geometric approach by developing the theory of deep level Deligne-Lusztig representations. It works for all residue fields, and supplies a large class of new (irreducible) supercuspidal representations.   

The notion of deep level Deligne-Lusztig representations was first introduced by Lusztig \cite{Lu79}, aiming to provide a geometric realization of G\'{e}rardin's construction of supercuspidals \cite{Ger75a}. Over the past dozen years, considerable progress has been made on this subject, including the Mackey formula, explicit comparison with supercuspidals, geometrization via character sheaves, and various applications to the local Langlands correspondence. We refer to \cite{BW16}, \cite{Chan20}, \cite{CI23}, \cite{CS23}, \cite{Iva23}, \cite{CO25a}, \cite{BC24}, \cite{Chan24}, \cite{Nie24}, \cite{IN25}, \cite{CO25b}, \cite{INY25}, \cite{IN26}, \cite{Tak26}, \cite{LN26} for the details.

\medskip

Let $G$ be a connected reductive group defined over a non-archimedean local field $k$ with residue characteristic $p$ and ring of integers $\CO_k$. Fix a point $\bx$ in the Bruhat-Tits building and denote by $\CG_{\bx, 0}$ the associated connected parahoric group scheme. Let $T \subseteq G$ be an unramified maximal torus whose apartment contains $\bx$. 

Similar to the classical Deligne-Lusztig theory \cite{DL76}, the deep level analog can be viewed as a functor which associates to each smooth character $\phi$ of a $p$-adic maximal unramified torus $T$ a virtual smooth representation $R(\phi, X)$ of $Z(k) \CG_{\bx, 0}(\CO_k)$, via the cohomological induction of a suitable subvariety $X$ of the (truncated) jet scheme of $\CG_{\bx, 0}$. Here $Z$ is the center of $G$. In the original construction of Lusztig, $X$ is chosen to be a straightforward extension of the classical Deligne-Lusztig variety \cite{Lu79, CI21}. This definition works for all primes $p$ but is hard to study in general. On the other hand, motivated by the pioneering work \cite{CS17, CS23}, the second author introduced another variety $Z_\phi$ under a mild condition on $p$ and proved that the new representation $R(\phi, Z_\phi)$ coincides with $R(\phi) = R(\phi, X)$ in the elliptic case. The behavior of this new construction is much easier to handle in various aspects such as decomposition into irreducible factors \cite{Nie24}, construction of character sheaves \cite{INY25} and distinction problems for symmetric pairs \cite{LN26}. However, almost all of the known results on $R(\phi, X)$, as well as the construction of $Z_\phi$, are built upon the assumption that $\phi$ admits a Howe factorization in the sense of Kaletha \cite{Kal}, which fails when $p$ is bad for $G$ or $p \mid |\pi_1(G_\der)|$.

In this note, we extend the definition of $Z_\phi$ and $\CR(\phi) = R(\phi, Z_\phi)$ for all primes $p$. The construction is based on the root filtration \[\tag{$\ast$} \emptyset \subsetneq \Phi^0 \subsetneq \Phi^1 \subsetneq \cdots \subsetneq \Phi^d = \Phi(G, T)\] associated to $\phi$ considered in \cite[\S 3.6]{Kal}. However, when $p$ is small, two difficulties arise. The first one is that some subsets in $(*)$ are no longer closed under addition. The second one is essential: there is a connected subgroup $\CT_\phi$ in the $\CO_k$-model $\CT$ for $T$ acting on the variety $Z_\phi$ but $\phi$ is non-trivial over $\CT_\phi(\CO_k)$, which implies that the $\phi$-isotypic part of the cohomology of $Z_\phi$ vanishes and hence $\CR(\phi) = 0$. Thus to produce nontrivial results, we impose the following condition \[\tag{\dag} \text{$\phi$ is trivial over $\CT_\phi(\CO_k)$}.\]

\begin{remark} \label{rmk:dag}
    We record two comments on the condition $(\dag)$. \begin{itemize}
        \item  The condition ($\dag$) holds true if $\phi$ admits a Howe factorization in the sense of Kaletha, which always exists if $p$ is not a bad prime for $G$ and $p \nmid |\p_1(G_\der)|$ by \cite[Proposition 3.6.7]{Kal}.

        \item When $q=2$, there are plenty of characters $\phi$ such that the condition ($\dag$) holds true but some subsets in $(*)$ are not closed under addition and hence the Howe factorization fails.
    \end{itemize} 
\end{remark}
%By abuse of notation, we also denote by $\phi$ its restriction to $T(k)\cap\CG_{\bx,0}(\CO_k)$ and the induced character of $T_r^F$ whenever defined. Throughout the paper, all Weyl group stabilizers of $\phi$ and all conditions $^z\phi=\phi$ refer to its restriction. The original character of $T(k)$ is used to specify the central action in the compact induction construction.

Our first main result is a dimension formula of virtual intertwining spaces. Here we 
\begin{theorem} \label{intro-Mackey}
    Assume $(\dag)$ holds. For $w \in G(k)$ we have 
    \[\dim I_w(\CR(\phi)) = |\{z \in (W_\bx z_w W_\bx)^F; {}^z \phi |_{\CT(\CO_k)} = \phi|_{\CT(\CO_k)}\}|,\] where $W_\bx$ is the (absolute) Weyl group of $\CT$ in $\CG_{\bx, 0}$, $F$ is the Frobenius map of $G$ and $z_w$ is the Iwahori-Weyl group element associated to $w$.
\end{theorem}
We refer to Theorem \ref{intertwine} for the precise statement of the above theorem. The proof is motivated by work \cite{LN26} on the distinction problem of deep level Deligne-Lusztig representations.

As an immediate consequence, we extend the Mackey inner product formula of $\CR(\phi)$ to the general case of all primes $p$.
\begin{corollary} \label{intro-irr}
    Assume $(\dag)$ holds. We have  \[\dim \Hom_{\CG_{\bx, 0}(\CO_k)}(\CR(\phi), \CR(\phi)) = |\{z \in W_\bx^F; {}^z \phi |_{\CT(\CO_k)} = \phi|_{\CT(\CO_k)}\}|.\] In particular,
    $\pm \CR(\phi)$ is irreducible if and only if the stabilizer of $\phi|_{\CT(\CO_k)}$ in $W_\bx^F$ is trivial.
\end{corollary}
The original restrictive proof of the Mackey formula in \cite{Nie24} relies on an algebraic intertwining result using the genericity condition (GE2) of \cite{Yu}, which fails for torsion primes $p$. The approach in this paper is purely geometric and works for all $p$.

Now we state the main results on the construction of supercuspidal representations of $G(k)$.
\begin{theorem} [Corollary \ref{direct-sum}] \label{intro-sum}
    Assume that $T$ is elliptic. If $\pm \CR(\phi)$ is a genuine representation, then \[\pi(\phi) := \pm \text{c-} \ind_{Z(k) \CG_{\bx, 0}(\CO_k)}^{G(k)} \CR(\phi)\] is a direct sum of irreducible supercuspidals.
\end{theorem}
Analogous partial results for $R(\phi)$ are obtained by \cite{CI23} in the case of inner forms of $\GL_n$ and by \cite{INT24} in the Coxeter case.

The following result produces an explicit construction of a large class of irreducible supercuspidals.
\begin{theorem} [Corollary \ref{irr-sup}] \label{intro-irr-sup}
    Assume $(\dag)$ holds and $T$ is elliptic. If the stabilizer of $\phi|_{\CT(\CO_k)}$ in $\pr(\tW_{[\bx]}^F)$ is trivial, then \[\text{$\pi(\phi)$ is an irreducible supercuspidal.}\] Here $[\bx]$ is the natural image of $\bx$ in the reduced building, $\tW_{[\bx]}$ is the stabilizer of $[\bx]$ in the Iwahori-Weyl group and $\pr(\tW_{[\bx]})$ denotes its natural image in the Weyl group.
\end{theorem}
The representations $\pi(\phi)$ in Theorem \ref{intro-sum} and Theorem \ref{intro-irr-sup} seem to be the geometric extensions of the construction by G\'{e}rardin \cite{Ger75a, Ger75b}. When $p$ is large, $\pi(\phi)$ is regular supercuspidal in the sense of Kaletha \cite{Kal}.

\medskip

The paper is organized as follows. In \S \ref{sec:pre} we recall basic notions on parahoric subgroups and root systems. In \S \ref{sec:filtration} and \S \ref{sec:deep-level-DL-rep}, we extend the definition of deep level Deligne-Lusztig representations to the general case of all primes $p$. In \S \ref{sec:Sigma} and \S\ref{sec:dim-inter} we prove Theorem \ref{intro-Mackey} and Theorem \ref{intro-irr-sup}. In \S \ref{sec:sum} we prove Theorem \ref{intro-sum} by following a strategy in \cite{CI23}.

\subsection*{Acknowledgement}
The second author would like to thank Alexander Ivanov for helpful comments and discussions.

\section{Preliminary} \label{sec:pre}
Let $k$ be a local non-Archimedean field with residue field $\BF_q$ of characteristic $p$. We write $\brk$ for the completion of a maximal unramified extension of $k$, and $F$ for the Frobenius automorphism of $\brk$ over $k$. We denote by $\CO_k$ and $\COk$ the integer rings of $k$ and $\brk$ respectively. Let $\varpi$ be a fixed uniformizer of $k$. We also fix a prime $\ell\neq p$.

\subsection{Positive loops}
Let $\Perf_{\ov\BF_q}$ denote the category of perfect $\ov\BF_q$-schemes. Let $\fkX$ be an $\COk$-scheme. The positive loop functor (also called jet scheme) of $\fkX$ is defined by
\[
L^+ \fkX \colon \Perf_{\ov\BF_q}^{\rm op} \to {\rm Sets}, \quad R \mapsto \fkX(W_{\COk}(R)),
\]
where $W_{\COk}(R) = R[[\varpi]]$ if ${\rm char}\, \brk > 0$ and $W_{\COk}(R) = \COk \otimes_{\breve \BZ_p} W(R)$ otherwise. Here $W(R)$ denotes the $p$-typical Witt-vectors of $R$. In this paper, we only consider schemes $\fkX$ which are affine and of finite type over $\COk$, in which case $L^+\fkX$ is representable by an affine perfect scheme over $\ov\BF_q$.

\subsection{Moy-Prasad quotients}
Let $G$ be a connected reductive group defined over $k$. Let $\bx$ be a point in the Bruhat-Tits building $\CB(G, k)$ of $G$ over $k$. Let $\CT$ be the corresponding integral model for $T$. Bruhat--Tits theory associates to each $0 \le r \in \widetilde \BR := \BR \sqcup \{s+; s \in \BR\}$ a Moy-Prasad $\CO_k$-group model of $G$. For any $0 \le s \le r \in \widetilde \BR$, $L^+(\CG_{\bx,r+} \times_{\CO_k} \CO_{\breve k}) \subseteq L^+(\CG_{\bx,s} \times_{\CO_k} \CO_{\breve k})$ is a subgroup, and we put
\[G_{s:r} = L^+(\CG_{\bx,s} \times_{\CO_k} \CO_{\breve k}) / L^+(\CG_{\bx,r+} \times_{\CO_k} \CO_{\breve k}),\]
which is a pfp perfectly smooth affine group scheme over $\ov\BF_q$. As the Frobenius automorphism $F$ fixes $\bx$, it acts naturally on $G_{s:r}$ and hence endows it with an $\BF_q$-rational structure. In particular, $G_{s:r}(\BF_q) = \CG_{\bx,s}(\CO_k)/\CG_{\bx,r+}(\CO_k)$.

Let $H \subseteq G$ be a closed subgroup. Following \cite[\S2.5]{CI21} one may construct an $\ov\BF_q$-rational closed subgroup $H_{s:r} \subseteq G_{s:r}$ in a similar way. We put $H_r = H_{0:r}$ for simplicity. If, moreover, $H$ is a $k$-rational subgroup, then $H_{s:r}$ is defined over $\BF_q$, and we still denote by $F$ the induced Frobenius automorphisms on $H$ or $H_{s:r}$. We put $H = H(\brk)$ and $H_{s:r} = H_{s:r}(\ov\BF_q)$ for simplicity.

\subsection{The maximal torus}
Let $T \subseteq G$ be a $k$-rational and $\brk$-split maximal torus such that $\bx \in \CA(T, \brk) \cap \CB(G, k)$. Here $\CA(T, \brk)$ denotes the apartment of $T$ over $\brk$. By fixing a special point $o \in \CA(T, \brk)$ fixed by the Frobenius of the unique quasi-split inner form of $G$, we identify $\CA(T, \brk)$ with the underlying affine space of the real vector space $X_*(T) \otimes \BR$. By this identification, the action of $F$ on $\CA(T, \brk)$ induces an affine transformation of $X_*(T) \otimes \BR$, which we still denote by $F$. 

Let $N_T$ be the normalizer of $T$ in $G$. We have the finite Weyl group $W = N_T(\brk) / T(\brk)$ and the Iwahori-Weyl group $\tW = N_T(\brk) / T(\COk)$. As the Frobenius map $F$ preserves $T$, it acts on $W$ and $\tW$ in a natural way. There is a natural action of $\tW \rtimes \<F\>$ on $X_*(T) \otimes \BR$ by affine transformations.

Let $\Phi = \Phi(G_\brk, T_\brk)$ denote the root system of $T_\brk$ in $G_\brk$. For $\a \in \Phi$ we write $G^\a \subseteq G$ and $\a^\vee: \BG_m \to T_\brk$ for the corresponding root subgroup and coroot respectively. We fix a normalized Chevalley system for $(G, T)$ compatible with the point $o$, which is a collection of isomorphisms $u_\a : \BG_a \to G^\a$ for $\a \in \Phi$ subject to certain properties, see \cite{BT2} or \cite{Ad}. 

Let $\tPhi = \tPhi_\re \sqcup\, \BZ$ be the set of affine roots, where $\tPhi_\re = \Phi \times \BZ$ denotes the set of real affine roots. Let $f \in \tPhi$. We write $\a_f \in \Phi \sqcup \{0\}$, $n_f \in \BZ$ such that $f = (\a_f,n_f) = \a_f + n_f$. We view $f$ as an affine function on $X_*(T) \otimes \BR$ given by $f(v) = \a_f(v) + n_f$. Then the action of $\tW \rtimes \<F\>$ on $X_*(T) \otimes \BR$ induces an action on $\tPhi$ by $x(f)(v) = f(x\i(v))$.

We fix a basis $(\l_i)_{i=1}^l$ of the cocharacter group $X_*(T)$ of $T$. If $f \in \BZ_{\ge 1}$, we put $\CG^f =\CT(1 + \varpi^f\CO_\brk)$ and define \[u_f: \BA^l \to \CG^f, \quad (x_i)_i \mapsto \prod_i \l_i(1 + \varpi^{n_f} [x_i]),\] where $[x_i] = x_i$ in the equal characteristic case and $[x_i]$ is the Teichm\"{u}ller lift of $x_i \in \ov\BF_q$ in the mixed characteristic case. If $f \in \tPhi_\re$, we put $\CG^f = u_{\a_f}(\varpi^{n_f} \CO_\brk)$ and define \[u_f: \BA^1 \to \CG^f, \quad x \mapsto u_{\a_f}(\varpi^{n_f} [x]).\] We put $G^f = \Im (u_f)$ for $f \in \tPhi^{>0}$. The parahoric subgroup $G_{\bx,0} := \CG_{\bx, 0}(\COk)$ attached to $\bx$ is generated by $\CT(\COk)$ and $\CG^f$ for $f \in \tPhi_\re$ such that $f(\bx) \ge 0$. Moreover, we have the Bruhat decomposition \[G = G(\brk) = \sqcup_{z \in W_\bx \backslash \tW / W_\bx} G_{\bx, 0} z G_{\bx, 0},\] where $W_\bx \subseteq \tW$ is the Weyl group of $G_{\bx, 0}$.

\medskip

Let $B = T U \subseteq G$ be a Borel subgroup with unipotent radical $U$. We denote by $\Phi^+$ the set of (positive) roots which appear in $B$. Let $\Pi = \{\a_1, \dots, \a_d\}$ be the set of simple roots of $\Phi^+$. For $\nu \in \BZ\Phi$ we write $\nu \prec 0$ if $\nu = c_1 \a_1 + \cdots + c_d \a_d \in \BZ\Phi$ such that there exists $1 \le i \le d$ such that $c_1 = \cdots = c_{i-1} = 0$ and $c_i < 0$. This defines a linear order $\preceq$ on $\BZ\Phi$ by \[\l \prec \l' \Leftrightarrow (\l-\l') \prec 0.\] Let $\leq$ be a linear order on $(\BZ\Phi) \times \BR$ such that $\l + n < \l' + n'$ if and only if either $\l(\bx) + n < \l'(\bx) + n'$ or $\l(\bx) + n = \l'(\bx) + n'$ and $\l \prec \l'$. Let $\tPhi^{>0}$ be the set of affine roots $f > 0$, that is, either $f(\bx) > 0$ or $f(\bx) = 0$ and $\a_f \in \Phi^+$.
\begin{lemma}
    Let $f, f', f_1, f_1' \in (\BZ\Phi) \times \BR$. Then we have 
    
    (1) $f < f'$ if and only if $-f > -f'$;

    (2) $f + f_1 < f' + f_1'$ if $f < f'$ and $f_1 < f_1'$.   
\end{lemma}
\begin{proof}
    It follows directly by definition. 
\end{proof}

\section{Filtration of the root system} \label{sec:filtration}
Let $\phi: T^F \to \ov\BQ_\ell^\times$ be a fixed smooth character of depth $r_\phi \in \BZ_{\ge 0}$. Let $r \ge r_\phi$. For $\a \in \Phi$ we define \[r_\a = \min\{i \in \BZ_{\ge 0}; \phi \circ N_F^{F^N} \circ \a^\vee(1 + \varpi^{i+1}\CO_\brk^{F^N}) = \{1\}\}.\] Here $N$ is any positive integer such that $F^N$ acts trivially on $X_*(T)$ and $N_F^{F^N}: T \to T$ is given by $x \mapsto x F(x) \cdots F^{N-1}(x)$. 

By definition, we have $r_{-\a} = r_\a = r_{F(\a)}$ for $\a \in \Phi$. Moreover, if $r_\a \ge 1$, then \[\phi \circ N_F^{F^N} \circ \a^\vee(1 + \varpi^{r_\a} \CO_\brk^{F^N}) = \phi \circ N_F^{F^N} \circ \a^\vee(1 + \varpi^{r_\a} [\BF_{q^N}]) \neq \{1\}.\] We set $s_\a = r_\a/2$. Put $r_f = 2 s_f  = r_{\a_f}$ if $f \in \tPhi_\re$ and $r_f = 2 s_f = -\infty$ if $f \in \BZ$.

\subsection{A result on $r_\a$} 
If ${\rm char} k = 0$ let $m \in \BZ_{\ge 1}$ such that $p \in \varpi^m \CO_k^\times$. Otherwise, we put $m = +\infty$.
\begin{lemma} \label{r-gamma}
    Let $\a, \b \in \Phi$ such that $\a + \b \neq 0$ and $\g = i \a + j \b \in \Phi$ for some $i, j \in \BZ_{\ge 1}$. Let $c \in \{1, 2, 3\}$ such that $c \g^\vee = a\a^\vee + b\b^\vee$ for some integers $a, b \in \BZ_{\ge 1}$ coprime to $c$. Then either $s_\g \le m$ or $r_\g - m \le \max\{r_\a, r_\b\}$.
\end{lemma}
\begin{proof}
    Let $n = \max\{r_\a, r_\b\} \in \BZ_{\ge 0}$. First assume that $c \neq p$. Then we have \begin{align*} &\quad\ \phi \circ N_F^{F^N} \circ \g^\vee(1 + \varpi^{n+1} \CO_\brk^{F^N}) \\ &= \phi \circ N_F^{F^N} \circ (c\g^\vee)(1 + \varpi^{n+1} [\BF_{q^N}]) \\ &\subseteq \phi \circ N_F^{F^N} \circ (a\a^\vee + b\b^\vee)(1 + \varpi^{n+1} [\BF_{q^N}])  = \{1\}.\end{align*} This implies that $r_\g \le n$ as desired. 
    
    Now we assume $c = p$ and $s_\g > m$. Let $\e \in \BF_q^\times$ such that $p - \varpi^m[\e] \in \varpi^{m+1} \COk$. For $x \in \ov\BF_q^\times$ we have \[(1 + \varpi^{r_\g -m }[x])^c = 1 + c \varpi^{r_\g - m} [x] + \sum_{i=2}^c \binom{c}{i} \varpi^{i(r_\g-m)} [x^i] \in 1 + \varpi^{r_\g}[\e x] + \varpi^{r_\g+1}\COk,\] where the inclusion follows from the fact that $i(r_\g-m) > i(r_\g- s_\g) = i s_\g \ge r_\g$ for $i \ge 2$. Therefore, we have \begin{align*} &\quad\ \phi \circ N_F^{F^N} \circ (a\a^\vee + b\b^\vee)(1 + \varpi^{r_\g - m} [\BF_{q^N}]) \\ &=\phi \circ N_F^{F^N} \circ (c\g^\vee)(1 + \varpi^{r_\g - m} [\BF_{q^N}]) \\ &= \phi \circ N_F^{F^N} \circ \g^\vee((1 + \varpi^{r_\g - m} [\BF_{q^N}])^c) \\ &= \phi \circ N_F^{F^N} \circ \g^\vee(1 + \varpi^{r_\g}[\BF_{q^N}]) \\ &\neq \{1\}. \end{align*} Hence $\phi \circ N_F^{F^N} \circ \a^\vee(1 + \varpi^{r_\g-m}[\BF_{q^N}]) \neq \{1\}$ or $\phi \circ N_F^{F^N} \circ \b^\vee(1 + \varpi^{r_\g-m}[\BF_{q^N}]) \neq \{1\}$, which means that  either $r_\a \ge r_\g - m$ or $r_\b \ge r_\g - m$ and the proof is finished.   
\end{proof}

\begin{theorem} \label{coefficient}
    Let $\a, \b \in \Phi$ with $\a + \b \neq 0$. For $x, y \in \COk$ we have \[ [u_\a(x), u_\b(y)] = \prod_{i, j \in \BZ_{\ge 1},\ i\a+j\b \in \Phi} u_{i\a+j\b}(C_{\a, \b, i, j} x^i y^j),\] where the coefficients $C_{\a, \b, i, j} \in \BZ$ are integers. Moreover, we have $c | C_{\a, \b, i, j}$, where $\a$, $\b$, $i$, $j$ and $c$ are as in Lemma \ref{r-gamma}.
\end{theorem}
\begin{proof}
    It follows directly from \cite[XXIII Proposition 6.4]{SGA3}.
\end{proof}

\begin{corollary} \label{comm}
    Let $f, f' \in \tPhi_\re^{>0}$ with $\a_f + \a_{f'} \neq 0$. Then $[\CG^f, \CG^{f'}]$ is contained in the product of subgroups $\CG^{f''}$ with $f'' = if + jf' + m_{f,f',i,j} \in \tPhi_\re^{>0}$ with $i, j \in \BZ_{\ge 1}$ and $m_{f,f',i,j} \in \BZ_{\ge 0}$ such that either $s_{f''} \le m_{f,f',i,j}$ or $r_{f''} - m_{f,f',i,j} \le \max\{r_f, r_{f'}\}$.
\end{corollary}
\begin{proof}
    Let $i, j \in \BZ_{\ge 1}$ such that $if +jf' \in \tPhi$. Let $m_{f, f', i, j} \in \BZ_{\ge 0}$ such that $C_{\a_f, \a_{f'},i,j} \in \varpi^{m_{f, f', i, j}} \CO_k^\times$ if it exists. Otherwise $C_{\a_f, \a_{f'},i,j}=0$ in $k$ and in this case there is no $u_{i\a_f+j\a_{f'}}$ term in the commutator. Then the statement follows directly from Theorem \ref{coefficient} and Lemma \ref{r-gamma}. 
\end{proof}

\subsection{A result on commutators} \label{subsec:commutator}
For $f \in \tPhi$ we define $f^\dag = r_f - f$. Note that for $f \in \BZ$ we have $f^\dag = -\infty$ and hence $f > f^\dag$.
\begin{lemma}
    For $f \in \tPhi_\re$ we have \[f > f^\dag \Leftrightarrow f > s_f \Leftrightarrow \text{either $f(\bx) > s_f$ or $f(\bx) = s_f$ and $\a_f > 0$}.\]    
\end{lemma}
\begin{proof}
    It follows directly by definition.
\end{proof}

Let $\CT_\phi \subseteq \CT(1+\varpi\COk)$ be the subgroup generated by $\CT(1 + \varpi^{r_\phi+1} \COk)$ and $\a^\vee(1 + \varpi^{r_\a + 1} \CO_\brk)$ for $\a \in \Phi$. Note that $\CT_\phi$ is $F$-stable.
\begin{lemma} \label{conjugate}
    Let $f, f' \in \tPhi^{>0}$ such that $\max\{f, f'\} > f^\dag, (f')^\dag$. Then $[\CG^f, \CG^{f'}]$ lies in a product of subgroups $\CT_\phi$ and $\CG^{f''}$ with $f+f' \le f'' \in \tPhi_\re^{>0}$ such that $(f'') > (f'')^\dag$.
\end{lemma}
\begin{proof}
    If $f, f' \in \BZ_{\ge 1}$, the statement is trivial. If $f \in \BZ$ and $f' \in \tPhi_\re$, then $[\CG^f, \CG^{f'}] \subseteq \CG^{f + f'}$, and the statement follows from the fact that \[(f + f')^\dag = r_{f'} - f' - f < r_{f'} - f' = (f')^\dag < \max\{f, f'\} < f + f'.\] By symmetry, it remains to consider the case that $f, f' \in \tPhi_\re$.
    
    Suppose $f + f' \in \BZ$. Then \[[\CG^f, \CG^{f'}] \subseteq \CG^{2f+f'} \a_{f'}^\vee(1+\varpi^{f+f'}\COk) \CG^{2f'+f}.\] If $f > (f')^\dag = r_{f'} - f'$, then $f + f' \ge r_{f'}+1$. Otherwise, we have $f' > (f')^\dag$ by assumption and hence $f'(\bx) \ge r_{f'}/2 = s_{f'}$. Similarly, we have $f(\bx) \ge s_f = s_{f'}$. Moreover, since one of $\a_f$ and $\a_{f'}$ is a negative root, it follows that $(f+f')(\bx) > 2 s_{f'} = r_{f'}$ as desired. Hence in either case we have $f + f' \ge r_{f'}+1$. Moreover, \[(2f+f')^\dag, (f+2f')^\dag < r_{f'} - (r_{f'} +1) = -1 < f + 2f', 2f + f'.\] Hence the statement follows.
    
    Suppose $\a_f + \a_{f'} \neq 0$. By Corollary \ref{comm}, \[[\CG^f, \CG^{f'}] \subseteq \prod_{f''= if + jf' + m_{f,f',i,j} \in \tPhi_\re} \CG^{f''},\] where  $i, j \in \BZ_{\ge 1}$ and $m_{f,f',i,j} \in \BZ_{\ge 0}$ such that either $s_{f''} \le m_{f,f',i,j}$ or $r_{f''} - m_{f,f',i,j} \le \max\{r_f, r_{f'}\}$. In the former case, we have \[(f'')^\dag = r_{f''} - f'' < r_{f''} - m_{f,f',i,j} \le s_{f''} \le m_{f,f',i,j} < f''\] as desired. In the latter case, we have \[(f'')^\dag \le r_{f''} - m_{f,f',i,j} - f - f' \le \max\{r_f, r_{f'}\} - f -f' < \max\{f^\dag, (f')^\dag\} < f + f' \le f''\] as desired.
\end{proof}

\section{The deep level Deligne-Lusztig representation} \label{sec:deep-level-DL-rep}
Let $\phi: T^F \to \ov\BQ_\ell^\times$ be a fixed smooth character of depth $r_\phi$. We keep the notation in \S \ref{sec:filtration} defined with respect to $\phi$.

\subsection{The group $\CI_\phi^\dag$} 
Recall that $G^f = \Im (u_f)$ for $f \in \tPhi^{>0}$ and $\CT_\phi$ is generated by $\CT(1 + \varpi^{r_\phi+1} \COk)$ and $\a^\vee(1+\varpi^{r_\a+1}\COk)$ for $\a \in \Phi$. We define \[ \CI_\phi^\dag = \CT_\phi \prod_{f \in \tPhi_\re^{>0}, \ f > f^\dag} G^f,\] where the product is taken with respect to the linear order $\le$ on $\tPhi$.

\begin{proposition} \label{I-group}
The set $\CI_\phi^\dag$ is a subgroup of $G_{\bx, 0}$.
\end{proposition}
\begin{proof}
    Let $f \in \tPhi_\re^{>0}$ such that $f > f^\dag$. We define \[\CI_\phi^{\dag, \ge f} = \CT_\phi \prod_{f \le f' \in \tPhi_\re^{>0},\ f' > (f')^\dag} G^{f'}.\] It suffices to show $\CI_\phi^{\dag, \ge f}$ is a subgroup.

    We argue by induction on $f$. If $f(\bx) > r_\phi$, then the statement follows since $\CI_\phi^{\dag, \ge f}$ is the subgroup generated by $\CT_\phi$ and $G^{f'}$ for $f \le f' \in \tPhi^{>0}$. Assume $f(\bx) \le r$. If $\{f' \in \tPhi_\re; f' > f,\ f' > (f')^\dag\}=\emptyset$, then the statement follows directly from the fact that $G^f$ is normalized by $\CT_\phi$. Otherwise let $$h=\min\{f' \in \tPhi_\re; f' > f,\ f' > (f')^\dag\}.$$ By the induction hypothesis, $\CI_\phi^{\dag, \ge h}$ is a subgroup. To show $\CI_\phi^{\dag, \ge f}$ is a subgroup, it suffices to show that $\CI_\phi^{\dag, \ge h}$ is normalized by $G^f$. Note that $[G^f, \CT_\phi] \subseteq \CG^{f+1} \subseteq \CI_\phi^{\dag, \ge h}$. It suffices to show  $[G^f, G^{f'}] \subseteq \CI_\phi^{\dag, \ge h}$ for all $f < f' \in \tPhi_\re$ such that $f' > (f')^\dag$. This follows from Lemma \ref{conjugate}, and the statement follows.    
\end{proof}

\begin{corollary} \label{commutator}
    Let $f_i \in \tPhi^{>0}$ be positive affine roots for $1 \le i \le n$. Suppose \[0 < h = \max\{f_i^\dag; f_i \in \tPhi_\re, f_i < f_i^\dag, 1 \le i \le n\} \text{ exists}.\] Let $h \le f \in \tPhi^{>0}$ such that $f > f^\dag$. 
    
    Let $g' = \prod_{i=1}^n u_{f_i}(y_i)$ and $g = u_f(x)$, where $y_i \in \ov\BF_q$ (resp. $x \in \ov\BF_q$) if $f_i \in \tPhi_\re$ (resp. $f \in \tPhi_\re$) and $y_i \in \ov\BF_q^l$ (resp. $x \in \ov\BF_q^l$) otherwise. Then \begin{itemize}
        \item if $f > h$, then $[g, g'] \in \CI_\phi^\dag$;

        \item if $f = h$, then $[g, g'] \in \a_h^\vee(1 + \varpi^{r_h} [x \sum_{f_i^\dag = h} y_i]) \CI_\phi^\dag$.
    \end{itemize}
\end{corollary}
\begin{proof}
    We argue by induction on $n$. If $n = 0$, the statement is trivial. Assume $n \ge 1$. If $(f_n)^\dag \le h < f$, by Lemma \ref{conjugate} we have \[\tag{a} u_{f_n}(y_n) g u_{f_n}(y_n)\i \in g \CT_\phi \prod_{f + f_n \le f'' \in \tPhi_\re, f'' > (f'')^\dag} G^{f''}.\] If $(f_n)^\dag = f$, then \[\tag{b} u_{f_n}(y_n) g u_{f_n}(y_n)\i \in g \a_h^\vee(1 + \varpi^{r_h} [x y_n]) \CT_\phi \prod_{f + f_n \le f'' \in \tPhi_\re, f'' > (f'')^\dag} G^{f''}.\]

    The first statement follows from (a) and the induction hypothesis. The second one follows from (b), the first statement (noticing that $r_h > h$) and the induction hypothesis.
\end{proof}

%\subsection{} Let $\tPhi_\phi =\{f \in \tPhi_\re; r_f =0, f(\bx) = 0\}$. For $f = \a + n$ we denote by $\s_f$ the corresponding affine reflection on $\tPhi$ given by $f' \mapsto f'$ if $f' \in \BZ$ and $f' \mapsto \s_{\a_f}(\a_{f'}) + n_{f'} - n_f\a_{f'}(\a_f^\vee)$ otherwise. Here for $\a \in \tPhi$ we denote by $\s_\a$ the corresponding reflection on $X^*(T)$ given by $\l \mapsto \l - \l(\a^\vee)\a$.

%\begin{lemma} \label{affine-reflection}Let $f' \in \tPhi^{>0}$ and $f \in \tPhi_\phi$. Then $\s_f(f')(\bx) = \bx$, and moreover, $r_{\s_f(f')} = r_{f'}$ if $f \in \tPhi_\re$.  \end{lemma}
%\begin{proof}The first statement follows directly from that $f(\bx) = 0$. Assume $f' \in \tPhi_\re$. Then \[\a_{\s_f(f')}^\vee = \s_{\a_f}(\a_{f'})^\vee = \s_{\a_f}(\a_{f'}^\vee) = \a_{f'}^\vee + \a_f(\a_{f'}^\vee) \a_f^\vee \in \a_{f'}^\vee + \BZ \a_f^\vee.\] As $r_f = 0$, we have $r_{\s_f(f')} = r_{f'}$ as desired.  \end{proof}

%By Lemma \ref{affine-reflection}, $\tPhi_\phi$ is a root system and we denote by $W_\phi$ the corresponding Weyl group generated by $\s_f$ for $f \in \tPhi$. We define \[\CK_{\phi, r} = \cup_{w \in W_\phi} \CH_{\phi, r} T_r \dw \CH_{\phi, r},\] where $\dw \in G_r$ is some/any lift of $w$.

%\begin{proposition}The set $\CK_{\phi, r}$ is a subgroup of $G_r$.\end{proposition}

\subsection{Deep level Deligne-Lusztig representations} \label{subsec:DL-rep}
 Let $G_r^f$, $\CT_{\phi, r}$, $\CI_{\phi, r}^\dag$, $\CG^f_r$, ... be the natural images of $G^f$, $\CT_\phi$, $\CI_\phi^\dag$, $\CG^f$, ... under the natural projection $G_{\bx, 0} \to G_r$. Let $L: G_r \to G_r$ be the Lang map given by $g \mapsto g\i F(g)$.

Now we introduce the notion of deep level Deligne-Lusztig representations. First consider the variety \[X_r = X_{\phi, \bx, r} := L\i (F\CI_{\phi, r}^\dag),\] on which $G_r^F \times T_r^F$ acts by $(g, t): x \mapsto g x t$. Then the $\phi$-isotypic part of the cohomology $ H_c^i(X_r, \ov\BQ_\ell)[\phi]$ is a representation of $G_r^F$. The virtual representation \[\CR_{T_r}^{G_r}(\phi) = H_c^*(X_r, \ov\BQ_\ell)[\phi] := \sum_{i \in \BZ} (-1)^i H_c^i(X_r, \ov\BQ_\ell)[\phi]\] is called a deep level Deligne-Lusztig representation of $G_r^F$. By inflation, we also view $\CR_{T_r}^{G_r}(\phi)$ as a virtual smooth representation of $G_{\bx, 0}^F$.

\begin{lemma} \label{independent}
    The representation $\CR_{T_r}^{G_r}(\phi)$ of $G_{\bx, 0}^F$ is independent of the choice of $r$.
\end{lemma}
\begin{proof}
    The proof is the same as that of \cite[Proposition 3.6]{Nie24}.
\end{proof}

\subsection{The intertwining space}
Fix $w \in G^F$ and let $r'\gg r$ be a sufficiently large integer such that ${}^{w\i} G_{\bx, r'+} \subseteq G_{\bx, r+}$. We set \[P_{r', r} = \{(x, w\i x w) \in G_{r'} \times G_r; x \in G_{\bx, 0} \cap {}^w G_{\bx, 0}\},\] which is a subgroup of $G_{r'} \times G_r$.
\begin{lemma}
    The subgroup $P_{r' ,r}$ is connected.
\end{lemma}
\begin{proof}
    It follows from the fact that $\CG_{\bx, 0} \cap {}^w \CG_{\bx, 0}$ is a connected group scheme. See \cite[Proposition 3 and Remark 4]{HR}.
\end{proof}

\begin{proposition} \label{identification}
    The map $(g_1, g_2) \mapsto g_1 w g_2\i$ induces an isomorphism \[\th: (G_{r'} \times G_r) / P_{r', r} \cong (G_{\bx, 0} w G_{\bx, 0}) / G_{\bx, r+}.\]
\end{proposition}
\begin{proof}
    Let $z = x w y G_{\bx, r+} \in (G_{\bx, 0} w G_{\bx, 0}) / G_{\bx, r+}$ for some $x, y \in G_{\bx, 0}$. Then the map $z \mapsto (x G_{\bx, r'+}, y\i G_{\bx, r+}) P_{r', r}$ gives a morphism \[\vartheta: (G_{\bx, 0} w G_{\bx, 0}) / G_{\bx, r+} \to (G_{r'} \times G_r) / P_{r', r}.\] One checks directly that $\th$ and $\vartheta$ are inverse to each other. 
\end{proof}

Let $\rho$ be a virtual smooth finite-dimensional representation of $G_{\bx, 0}^F$. We denote by ${}^w\rho$ the representation of $G_{\bx, 0}^F \cap {}^w G_{\bx, 0}^F$ given by ${}^w\rho(x) = \rho(w\i x w)$. The intertwining space by $w$ is defined as \[I_w(\rho) = \Hom_{G_{\bx, 0}^F \cap {}^w G_{\bx, 0}^F}(\rho, {}^w \rho).\]
\begin{lemma} \label{finite}
    We have \[\dim I_w(\CR_{T_r}^{G_r}(\phi)) = \dim (\CR_{T_{r'}}^{G_{r'}}(\phi) \boxtimes \CR_{T_r}^{G_r}(\phi\i))^{P_{r', r}^F}.\]
\end{lemma}
\begin{proof}
    By the choice of $r'$, the virtual modules $\CR_{T_r}^{G_r}(\phi)$ and ${}^w\CR_{T_r}^{G_r}(\phi)$ of $G_{\bx, 0}^F \cap {}^w G_{\bx, 0}^F$ descend to virtual modules of $(G_{\bx, 0} \cap {}^w G_{\bx, 0})^F / G_{\bx, r'+}^F$. Let $H$ be the graph of the group homomorphism \[(G_{\bx, 0} \cap {}^w G_{\bx, 0})^F / G_{\bx, r'+}^F \to ({}^{w\i}G_{\bx, 0} \cap G_{\bx, 0})^F / {}^{w\i} G_{\bx, r'+}^F\] given by $x \mapsto w\i x w$. Then we have \begin{align*}
        &\quad\ \dim I_w(\CR_{T_r}^{G_r}(\phi)) \\ &=\dim \Hom_{G_{\bx, 0}^F \cap {}^w G_{\bx, 0}^F}(\CR_{T_r}^{G_r}(\phi), {}^w\CR_{T_r}^{G_r}(\phi))  \\  &=\dim \Hom_{G_{\bx, 0}^F \cap {}^w G_{\bx, 0}^F}(\CR_{T_{r'}}^{G_{r'}}(\phi), {}^w\CR_{T_r}^{G_r}(\phi)) \\ &=\dim \Hom_{(G_{\bx, 0}^F \cap {}^w G_{\bx, 0}^F)/G_{\bx, r'+}^F}(\CR_{T_{r'}}^{G_{r'}}(\phi), {}^w\CR_{T_r}^{G_r}(\phi)) \\ &=\<\CR_{T_{r'}}^{G_{r'}}(\phi), {}^w\CR_{T_r}^{G_r}(\phi)\>_{(G_{\bx, 0}^F \cap {}^w G_{\bx, 0}^F)/G_{\bx, r'+}^F} \\ &= \dim (\CR_{T_{r'}}^{G_{r'}}(\phi) \boxtimes \CR_{T_r}^{G_r}(\phi)^*)^H \\ & = \dim (\CR_{T_{r'}}^{G_{r'}}(\phi) \boxtimes \CR_{T_r}^{G_r}(\phi\i))^H \\ &= \dim (\CR_{T_{r'}}^{G_{r'}}(\phi) \boxtimes \CR_{T_r}^{G_r}(\phi\i))^{P_{r', r}^F},    
    \end{align*} where the second identity follows from Lemma \ref{independent} and the last equality follows from the fact that the image of $H$ under the natural projection $((G_{\bx, 0} \cap {}^w G_{\bx, 0})^F / G_{\bx, r'+}^F) \times (({}^{w\i}G_{\bx, 0}^F \cap G_{\bx, 0})^F / {}^{w\i} G_{\bx, r'+}^F) \to G_{r'}^F \times G_r^F$ is $P_{r', r}^F$.   
\end{proof}

Consider the variety \[\Sigma =\{(\g, \psi) \in (F\CI_{\phi, r'}^\dag \times F\CI_{\phi, r}^\dag) \times (G_{r'} \times G_r)/P_{r', r}; F(\psi) = \g\i \ps\},\] on which $T_{r'}^F \times T_r^F$ acts by $\t: (\g, \psi) \mapsto (\t \g \t\i, \t\psi)$.
\begin{proposition} \label{geo}
    We have \[\dim I_w(\CR_{T_r}^{G_r}(\phi)) = \dim (\CR_{T_{r'}}^{G_{r'}}(\phi) \boxtimes \CR_{T_r}^{G_r}(\phi\i))^{P_{r', r}^F} = \dim H_c^*(\Sigma, \ov\BQ_\ell)[\phi \boxtimes \phi\i].\]
\end{proposition}
\begin{proof}
    It follows from Lemma \ref{finite} and Proposition \ref{quotient} below. 
\end{proof}

\begin{proposition} \label{quotient}
    Let $M$ be a linear algebraic group over $\BF_q$ with Frobenius map $F$. Let $H, N \subseteq M$ be two closed subgroups such that $N = FN$ is connected. Then the map $m \mapsto (L(m), m\i N)$ induces an isomorphism \[\k: N^F \backslash L\i(FH) \overset \sim \to \{(x, \psi) \in FH \times M/N; F(\psi) = x\i \psi\},\] where $L: M \to M, \ m \mapsto m\i F(m)$ denotes the Lang map.
\end{proposition}
\begin{proof}
   First note that $\k$ is injective. We show it is surjective. Let $(x, \psi) \in FH \times M/N$ such that $F(\psi) = x\i\psi$. Suppose that $\psi = z N$ for some $z \in M$. Then the equality $F(\psi) = x\i\psi$ is equivalent to the condition that $z\i x F(z) \in N$. As $N = FN$ is connected, there exists $u \in N$ such that $z\i x F(z) = u\i F(u)$. Set $m = u z\i$. Then $m\i N = z N$ and $L(m) = z u\i F(u) F(z\i) = x$, that is, $\k(N^F m) = (x, \psi)$. Hence $\k$ is surjective as desired.  
\end{proof}

\section{Cohomological vanishing results} \label{sec:Sigma}

Let the notation be as in previous sections. In this section, we deduce cohomological vanishing results on the variety $\Sigma$, which will be used to compute the dimension of the isotypic component $H_c^*(\Sigma, \ov\BQ_\ell)[\phi \boxtimes \phi\i]$ appearing in Proposition \ref{geo}.

Let $N_T$ be the normalizer of $T$ in $G$. Denote by $W = N_T(\brk) / T(\brk)$ and $\tW = N_T(\brk) / \CT = X_*(T) \rtimes W$ the Weyl group and the Iwahori-Weyl group of $G$ respectively. Recall the Bruhat decompositions \[G = \bigsqcup_{x \in \tW} \CI x \CI, \quad  G_{\bx, 0} = \bigsqcup_{x \in W_\bx} \CI x \CI,\] where $W_\bx \subseteq \tW$ denotes the Weyl group of $G_{\bx, 0}$, and $\CI$ is the Iwahori subgroup generated by $\CT$ and $G^f$ for $f \in \tPhi_\re^{>0}$.

For $w \in G^F$ we have \[G_{\bx, 0} w G_{\bx, 0} = \bigsqcup_{z \in W_\bx z_w W_\bx} \CI z \CI,\] where $z_w \in \tW$ satisfies $w \in \CI z_w \CI$.

Using the isomorphism $\th$ in Proposition \ref{identification}, we have a decomposition \[\Sigma = \bigsqcup_{z \in W_\bx z_w W_\bx} \Sigma_z,\] where \begin{align*} \Sigma_z &= \{(\g, \psi) \in \Sigma; \th(\psi) \in \CI z \CI / G_{\bx, r+}\} \\ &\cong \{(x_1, x_2, x) \in \CI_{\phi, r'}^\dag \times \CI_{\phi, r}^\dag  \times (\CI z \CI / G_{\bx, r+}); F(x_1 x x_2\i) = x\}\end{align*} on which $T_{r'}^F \times T_r^F$ acts by $(\t_1, \t_2): (x_1, x_2, x) \mapsto (\t_1 x_1 \t_1\i, \t_2 x_2 \t_2\i, \t_1 x \t_2\i)$.

\subsection{The subset $\Sigma_z$}
We fix $z \in W_\bx z_w W_\bx$. 
\begin{lemma} \label{value}
    Let $f \in \tPhi_\re$. We have $z(f)(\bx) - f(\bx) = \a_f(z\i(\bx) - \bx)$. In particular, $z(f)(\bx) \le f(\bx) \Leftrightarrow z(f^\dag)(\bx) \ge f^\dag(\bx)$.
\end{lemma}
\begin{proof}
    Assume $z = \l \rtimes u$ with $\l \in X_*(T)$ and $u \in W$. Then $z(f) = u(\a_f) + n_f + u(\a_f)(\l)$ and $z\i(\bx) = u\i(\bx) + u\i(\l)$. Thus \[z(f)(\bx) = u(\a_f)(\bx) + n_f + u(\a_f)(\l) = f(z\i(\bx)).\] Hence $z(f)(\bx) = f(z\i(\bx)) = f(\bx) + \a_f(z\i(\bx) - \bx)$ and the first statement follows.

    As $\a_{f^\dag} = -\a_f$, we have \[z(f^\dag)(\bx) \le f^\dag(\bx) \Leftrightarrow \a_{f^\dag}(z\i(\bx) - \bx) \le 0 \Leftrightarrow \a_f(z\i(\bx) - \bx) \ge 0 \Leftrightarrow z(f)(\bx) \ge f(\bx).\] Hence the second statement follows.
\end{proof}

Let $\tPhi_{\re, t}^{>0} = \{f \in \tPhi_\re^{>0}; 0 \le f(\bx) \le t\}$ for $t \in \BR_{\ge 0}$. We define \begin{align*}
    A &=\{f \in \tPhi_\re^{>0}; \text{ $z\i(f)(\bx) \le r$, either $z\i(f)(\bx) < f(\bx)$ or $z\i(f)(\bx) = f(\bx)$ and $\a_f > 0$}\}; \\ B &=\{f \in \tPhi_{\re, r}^{>0}; \text{ either $z(f)(\bx) < f(\bx)$ or $z(f)(\bx) = f(\bx)$ and $\a_{z(f)} < 0$}\}; \\ C &= \{f \in \tPhi_\re^{>0}; z\i(f) < 0\}.
\end{align*}

For $D \subseteq \tPhi_{\re, t}^{>0}$ we put $G_t^D = \prod_{f \in D} G_t^f$, where the product is taken with respect to the linear order $\le$ on $(\BZ\Phi) \times \BR$. For $v \in G_t^D$ there are unique elements $v_f \in G_t^f$ such that $v = \prod_{f \in D} v_f$. By abuse of notation, we also view $v_f$ as elements in $\ov\BF_q$ via the bijections $\BA^1 \overset {u_f} \to G_t^f$.
\begin{lemma} \label{coordinate}
    There is an isomorphism \[ G_{r'}^A \times zT_r \times G_r^B \to \CI z \CI / G_{\bx, r+},\quad\ (v_1, \dz, v_2) \mapsto v_1 \dz v_2\i.\]
\end{lemma}
\begin{proof}
    First note that we have an isomorphism \[G_{r'}^C \times zT_r \times G_r^{\tPhi_{\re,r}^{>0}} \to \CI z \CI / G_{\bx, r+},\quad\ (v_1, \dz, v_2) \mapsto v_1 \dz v_2\i.\] Let $D$ be the set of affine roots $f \in \tPhi_{\re, r}^{>0}$ such that either $\a_f(z\i(\bx) -\bx) > 0$ or $\a_f(z\i(\bx) -\bx) = 0$ and $\a_{z(f)} > 0$. By Lemma \ref{value}, we have $B = \tPhi_{\re, r}^{>0} \sm D$ and $A = z(D) \sqcup C$. Hence the statement follows.
\end{proof}

We put \[A' = \{f \in A; f < f^\dag\};\quad B' = \{f \in B; f < f^\dag\}.\] Using the isomorphism $\CI z \CI / G_{\bx, r+} \cong G_{r'}^A \times z T_r \times G_r^B$ of Lemma \ref{coordinate}, we have a decomposition \begin{align*}&\quad\ \CI z \CI / G_{\bx, r+} \\ &= (\CI z \CI / G_{\bx, r+})^0 \sqcup \bigsqcup_{(f_1, f_2) \in A' \times B'} (\CI z \CI / G_{\bx, r+})^{f_1, f_2} \sqcup \bigsqcup_{f_1 \in A'} (\CI z \CI / G_{\bx, r+})^{f_1, *} \sqcup \bigsqcup_{f_2 \in B'} (\CI z \CI / G_{\bx, r+})^{*, f_2},\end{align*} where $(\CI z \CI / G_{\bx, r+})^{f_1, f_2}$ (resp. $(\CI z \CI / G_{\bx, r+})^{f_1, *}$, resp. $(\CI z \CI / G_{\bx, r+})^{*, f_2}$) is the set of elements $v_1 \dz v_2\i \in \CI z \CI / G_{\bx, r+}$ with $v_1 \in G_{r'}^A$ and $v_2 \in G_r^B$ such that $(v_1)_{f_1},    (v_2)_{f_2} \neq 0$ (resp. $(v_1)_{f_1} \neq 0$, resp. $(v_2)_{f_2} \neq 0$) and $(v_1)_{h_1} = 0 = (v_2)_{h_2}$ for any $h_1 \in A'$ with $h_1^\dag > f_1^\dag$ and any $h_2 \in B'$ with $h_2^\dag > f_2^\dag$ (resp. for any $h_1 \in A'$ with $h_1^\dag > f_1^\dag$ and any $h_2 \in B'$, resp. for any $h_1 \in A'$ and any $h_2 \in B'$ with $h_2^\dag > f_2^\dag$).

Then we have an induced decomposition \[\Sigma_z = \Sigma_z^0 \sqcup \bigsqcup_{(f_1, f_2) \in A' \times B'} \Sigma_z^{f_1, f_2} \sqcup \bigsqcup_{f_1 \in A'} \Sigma_z^{f_1, *} \sqcup \bigsqcup_{f_2 \in B'} \Sigma_z^{*, f_2}.\]

The main result of this section is as follows.
\begin{proposition} \label{red}
     Let $f_1\in A'$ and $f_2 \in B'$. Then \[H_c^*(\Sigma_z^{f_1, f_2}, \ov\BQ_\ell)[\phi \otimes \phi\i] = H_c^*(\Sigma_z^{f_1, *}, \ov\BQ_\ell)[\phi \otimes \phi\i] = H_c^*(\Sigma_z^{*, f_2}, \ov\BQ_\ell)[\phi \otimes \phi\i] =0.\] In particular, $H_c^*(\Sigma_z, \ov\BQ_\ell)[\phi \otimes \phi\i] = H_c^*(\Sigma_z^0, \ov\BQ_\ell)[\phi \otimes \phi\i]$.
\end{proposition}

The proof occupies the remainder of the section. We will need the following result.
\begin{lemma} \label{gedag}
    Let $f_1 \in A'$ and $f_2 \in B'$. If $r_{f_1} \ge r_{z\i(f_1)}$ (resp. $r_{f_2} \ge r_{z(f_2)}$), then $z\i(f_1^\dag) > s_{z\i(f_1)}$ (resp. $z(f_2^\dag) > s_{z(f_2)}$). 
\end{lemma}
\begin{proof} 
    As $f_1 < f_1^\dag$ we have (a) either $f_1(\bx) < s_{f_1}$ or $f_1^\dag(\bx) = s_{f_1^\dag} = s_{f_1}$ and $0 < \a_{f_1^\dag} = - \a_{f_1}$. Moreover, by the definition of $A$ and Lemma \ref{value}, we have (b) either $z\i(f_1^\dag)(\bx) > f_1^\dag(\bx)$ or $z\i(f_1^\dag)(\bx) = f_1^\dag(\bx)$ and $\a_{f_1} > 0$. Suppose $r_{f_1} \ge r_{z\i(f_1)}$. Then \[z\i(f_1^\dag)(\bx) \ge f_1^\dag(\bx) \ge s_{f_1} \ge s_{z\i(f_1)}.\] Assume the equality holds (otherwise we have $z\i(f_1^\dag) > s_{z\i(f_1)}$ as desired). Then $z\i(f_1^\dag)(\bx) = f_1^\dag(\bx) = s_{f_1}$ and hence $f_1(\bx) = r_{f_1} - f_1^\dag(\bx) = s_{f_1}$. We have $\a_{f_1} < 0$ and $\a_{f_1} > 0$ by (a) and (b) respectively, which is a contradiction. So $z\i(f_1^\dag) > s_{z\i(f_1)}$ as desired.
    
    Similarly, we have (c) either $z(f_2^\dag)(\bx) > f_2^\dag(\bx)$ or $z(f_2^\dag)(\bx) = f_2^\dag(\bx)$ and $\a_{z(f_2^\dag)} = -\a_{z(f_2)} > 0$. Suppose that $r_{f_2} \ge r_{z(f_2)}$. Then \[z(f_2^\dag)(\bx) \ge f_2^\dag(\bx) \ge s_{f_2} \ge s_{z(f_2)}.\] If equality holds, it follows by (c) that $\a_{z(f_2^\dag)}  > 0$. So we always have $z(f_2^\dag) > s_{z(f_2)}$ as desired.
\end{proof}

\subsection{Vanishing results} 
In this subsection, we deduce certain vanishing results, which form the technical heart in the proof of Proposition \ref{red}.

\begin{lemma} \label{T-stable}
    The subsets $(\CI z \CI / G_{\bx, r+})^0$, $(\CI z \CI / G_{\bx, r+})^{f_1, f_2}$, $(\CI z \CI / G_{\bx, r+})^{f_1, *}$ and $(\CI z \CI / G_{\bx, r+})^{*, f_2}$ are $T_{r'} \times T_r$-invariant.
\end{lemma}
\begin{proof}
    Let $h_1, h_2 \in A$ such that $h_i^\dag \le \max\{h_i, f_1^\dag\}$ for $i \in \{1, 2\}$. Suppose $h_3 = i h_1 + j h_2 + m_{h_1, h_2, i, j} \in \tPhi_\re$ for some $i, j \in \BZ_{\ge 1}$. It suffices to show that $h_3^\dag < \max\{h_3, f_1^\dag\}$. 

    Indeed, we have either $s_{h_3} \le m_{h_1, h_2, i, j}$ or $r_{h_3} - m_{h_1, h_2, i, j} \le \max\{r_{h_1}, r_{h_2}\}$. If the former case occurs, then $h_3 \ge h_1(\bx) + h_2(\bx) + m_{h_1, h_2, i, j} > s_{h_3}$, that is, $h_3 > h_3^\dag$ as desired. If the latter case occurs, then \begin{align*}h_3^\dag &= r_{h_3} - m_{h_1, h_2, i, j} - i h_1 - j h_2 \le \max\{r_{h_1}, r_{h_2}\} - h_1 - h_2 \\  &< \max\{h_1^\dag, h_2^\dag\} \le \max\{h_1+h_2, f_1^\dag\} \le \max\{h_3, f_1^\dag\}.\end{align*} The proof is finished.
\end{proof}

For $\l \in X_*(T)$ and $m \in \BZ_{\ge 1}$ we define \[N_{\l, m}: \BA^1 \to T(\CO_\brk), \quad x \mapsto N_F^{F^N} \circ \l(1 + \varpi^m [x]).\] Note that $N_{\l, m}(\BF_{q^N}) \subseteq T(\CO_k)$.

\begin{lemma} \label{deform}
    Let $f_1\in A'$ and $f_2 \in B'$. Assume $z\i(f_1^\dag) > s_{z\i(f_1)}, f_2^\dag$ (resp. $z(f_2^\dag) > s_{z(f_2)}, f_1^\dag$) or $r_{z(f_2)} > r_{f_2}, f_1^\dag$ (resp. $r_{z\i(f_1)} > r_{f_1}, f_2^\dag$ ). Then there exists a morphism \[\D: \BA^1 \times \Sigma_z^{f_1, f_2} \to \Sigma_z^{f_1, f_2}\] such that \begin{itemize}
        \item the restriction morphisms $\D_\t = \D(\t, -): \Sigma_z^{f_1, f_2} \to \Sigma_z^{f_1, f_2}$ for $\t \in \BA^1$ are isomorphisms;

        \item the morphism $\D_\t$ coincides with the action of $(N_{\a_{f_1}^\vee, r_{f_1}}(\t), 1)$ (resp. $(1, N_{\a_{f_2}^\vee, r_{f_2}}(\t))$) or $(N_{\a_{z(f_2)}^\vee, r_{z(f_2)}}(\t), 1)$ (resp. $(1, N_{\a_{z\i(f_1)}^\vee, r_{z\i(f_1)}}(\t))$) on $\Sigma_z^{f_1, f_2}$ for $\t \in \BF_{q^N}$.
    \end{itemize}

    As a consequence, we have $H_c^*(\Sigma_z^{f_1, f_2}, \ov\BQ_\ell)[\phi \otimes \phi\i] = 0$.
\end{lemma}
\begin{proof}
    Suppose the morphism $\D$ is defined in the following way \[(\t, (x_1, x_2, v_1 \dz v_2\i)) \mapsto (x_1', x_2', F(t_1) v_1 \dz v_2\i F(t_2\i)),\] where $x_1' \in \CI_{\phi, r'}^\dag$, $x_2' \in \CI_{\phi, r}^\dag$, $t_1 \in T_{r'}$ and $t_2 \in T_r$. In view of Lemma \ref{T-stable}, to prove $\D$ is well-defined we need to show that \[\tag{a} F(t_1) v_1 \dz v_2\i F(t_2\i) = F(x_1' F(t_1) v_1 \dz v_2\i F(t_2\i) {x_2'}\i).\] As $(x_1, x_2, v_1 \dz v_2\i) \in \Sigma_z$, we have \[ F(t_1) v_1 \dz v_2\i F(t_2\i) = F(t_1) F(x_1 v_1 \dz v_2\i x_2\i)  F(t_2\i) = F(t_1 x_1 v_1 \dz v_2\i x_2\i t_2\i).\] Thus (a) is equivalent to \[\tag{b} x_1' F(t_1) v_1 \dz v_2\i F(t_2\i) {x_2'}\i = t_1 x_1 v_1 \dz v_2\i x_2\i t_2\i.\] In particular, (b) follows from the following two equalities \[x_1' F(t_1) v_1 = t_1 x_1 v_1 \d; \quad\ x_2' F(t_2) v_2 = t_2 x_2 v_2 {}^{\dz\i} \d,\] where $\d \in G_{\bx, 0}$ is a suitable element only depending on $\t$ and $v_1 \dz v_2\i$. Then it is sufficient to show the inclusions \[\tag{c} x_1' = {}^{t_1}(x_1 v_1 \d v_1\i L(t_1)\i) \in \CI_{\phi, r'}^\dag;\quad x_2' = {}^{t_2}(x_2 v_2 {}^{\dz\i}\d v_2\i L(t_2)\i) \in \CI_{\phi, r}^\dag,\] that is, \[\tag{d}  v_1 \d v_1\i L(t_1)\i \in \CI_{\phi, r'}^\dag;\quad v_2 {}^{\dz\i}\d v_2\i L(t_2)\i \in \CI_{\phi, r}^\dag.\] Moreover, by (c) we have \[\tag{e} \text{ if $\d = 1$ and $t_1, t_2$ are $F$-fixed for $\t \in \BF_{q^N}$, then $\D_\t$ coincides with the action of $(t_1, t_2)$.}\]

    Assume $z\i(f_1^\dag) > s_{z\i(f_1)}, f_2^\dag$. Let $\d= u_{f_1^\dag}((v_1)_{f_1}\i (F^N(\t) - \t))$, $t_1 = N_{\a_{f_1}^\vee, r_{f_1}}(\t)$ and $t_2 = 1$. Then $\d \in \CI_{\phi, r'}^\dag$ and ${}^{\dz\i} \d \in \CI_{\phi, r}^\dag$ since $f_1^\dag > f_1$ and $z\i(f_1^\dag) > s_{z\i(f_1)}$ respectively. As $z\i(f_1^\dag) > f_2^\dag$, it follows from Corollary \ref{commutator} that \[v_2 {}^{\dz\i} \d v_2\i L(t_2)\i = [v_2, {}^{\dz\i} \d] {}^{\dz\i} \d  \in \CI_{\phi, r}^\dag\] and \[v_1 \d v_1\i L(t_1)\i = \d [\d\i, v_1] L(t_1)\i \in \a_{f_1}^\vee(1 + \varpi^{r_{f_1}}[F^N(\t) - \t]) L(t_1)\i \CI_{\phi, r'}^\dag \subseteq \CI_{\phi, r'}^\dag.\] Hence (d) is satisfied and $\D$ is well-defined. Note that $\d = 1$ and $(t_1, t_2) = (N_{\a_{f_1}^\vee, r_{f_1}}(\t), 1)$ for $\t \in \BF_{q^N}$. The second property of $\D$ follows from (e). It remains to show $\D_\t$ is invertible for $\t \in \BA^1$. Indeed, we can define another morphism \[\D^\flat: \BA^1 \times \Sigma_z^{f_1, f_2} \to \Sigma_z^{f_1, f_2}\] in the above way by taking $\d= u_{f_1^\dag}((v_1)_{f_1}\i (F^N(\t) - \t))$, $t_1 = N_{\a_{f_1}^\vee, r_{f_1}}(-\t)\i$ and $t_2 = 1$. By (c), the morphism $\D^\flat_{-\t} \circ \D_\t$ is given by \[(x_1, x_2, v_1 \dz v_2\i) \mapsto (x_1 \d_1, x_2 \d_2, v_1 \dz v_2\i),\] where $\d_1 \in \CI_{\phi, r'}^\dag$ and $\d_2 \in \CI_{\phi, r}^\dag$ are certain elements only depending on $\t$ and $v_1 \dz v_2\i$. By definition, we have \[F(x_1 \d_1 v_1 \dz v_2\i \d_2\i x_2\i) = v_1 \dz v_2\i = F(x_1 v_1 \dz v_2\i x_2\i).\] In particular, $\d_1 v_1 \dz v_2\i \d_2\i = v_1 \dz v_2\i$. Then the map \[(x_1, x_2, v_1 \dz v_2\i) \mapsto (x_1 \d_1\i, x_2 \d_2\i, v_1 \dz v_2\i)\] is the inverse of $\D_{-\t}^\flat \circ \D_\t$. Hence $\D_\t$ has a left inverse. By symmetry, $\D_\t \circ \D_{-\t}^\flat$ is invertible and hence $\D_\t$ has a right inverse. Thus $\D_\t$ is invertible as desired. As in the proof of \cite[Proposition 6.4]{DL76}, the induced action of $\D_\t$ on each $H_c^i(\Sigma_z^{f_1, f_2})$ is trivial for $\t \in \BA^1$. On the other hand, by definition $\phi \otimes \phi\i$ is non-trivial over $(N_{\a_{f_1}^\vee, r_{f_1}}(\BF_{q^N}), 1)$. Thus $H_c^i(\Sigma_z^{f_1, f_2}, \ov\BQ_\ell)[\phi \otimes \phi\i] = 0$ as desired.

    Assume $z(f_2^\dag) > s_{z(f_2)}, f_1^\dag$. We take $\d = {}^\dz u_{f_2^\dag}((v_2)_{f_2}\i(F^N(\t) - \t))$, $t_1 = 1$ and $t_2 = N_{\a_{f_2}^\vee, r_{f_2}}(\t)$. Then ${}^{\dz\i}\d \in \CI_{\phi, r}^\dag$ and $\d \in \CI_{\phi, r'}^\dag$ since $f_2^\dag > f_2$ and $z(f_2^\dag) > s_{z(f_2)}$ respectively. As $z(f_2^\dag) > f_1^\dag$, it follows from Corollary \ref{commutator} that \[v_1 \d v_1\i L(t_1)\i = [v_1, \d] \d \in \CI_{\phi, r'}^\dag\] and \[v_2 {}^{\dz\i}\d v_2\i L(t_2)\i = {}^{\dz\i}\d [{}^{\dz\i}\d\i, v_2] L(t_2)\i \in \a_{f_2}^\vee(1 + \varpi^{r_{f_2}}[F^N(\t) - \t]) L(t_2)\i \CI_{\phi, r}^\dag \subseteq \CI_{\phi,r}^\dag.\] The remaining arguments follow as in the previous case.

    Assume $r_{z(f_2)} > r_{f_2}, f_1^\dag$ (resp.  $r_{z\i(f_1)} > r_{f_1}, f_2^\dag$). We take $(t_1, t_2, \d) = (N_{\a_{z(f_2)}^\vee, r_{z(f_2)}}(\t), 1, L(t_1))$ (resp. $(t_1, t_2, \d) = (1, N_{\a_{z\i(f_1)}^\vee, r_{z\i(f_1)}}(\t), {}^\dz L(t_2))$). As $r_{z(f_2)} > r_{f_2}$ (resp. $r_{z\i(f_1)} > r_{f_1}$), we have ${}^{\dz\i}\d \in \a_{f_2}^\vee(1 + \varpi^{r_{z(f_2)}} \CO_\brk) \subseteq \CT_{\phi, r}$ (resp. $\d \in \a_{f_1}^\vee(1 + \varpi^{r_{z\i(f_1)}} \CO_\brk) \subseteq \CT_{\phi, r'}$). As $r_{z(f_2)} > f_1^\dag$ (resp. $r_{z\i(f_1)} > f_2^\dag$), it follows by Corollary \ref{commutator} that \[\text{ $v_1 \d v_1\i L(t_1)\i = [v_1, L(t_1)] \in \CI_{\phi, r'}^\dag$ (resp. $v_2 {}^{\dz\i}\d v_2\i L(t_2)\i = [v_2, L(t_2)] \in \CI_{\phi, r}^\dag$) }.\] Moreover, as $r_{z(f_2)} > r_{f_2} > f_2^\dag$ (resp. $r_{z\i(f_1)} > r_{f_1} > f_1^\dag$), by Corollary \ref{commutator} we have \[\text{$v_2 {}^{\dz\i}\d v_2\i L(t_2)\i =  [v_2, {}^{\dz\i}\d] {}^{\dz\i}\d \in \CI_{\phi, r}^\dag$ (resp. $v_1 \d v_1\i L(t_1)\i = [v_1, \d] \d \in \CI_{\phi, r'}^\dag$)}.\] The remaining arguments follow as in the previous cases.  
\end{proof}

The following result is proved in the same way as Lemma \ref{deform}.
\begin{lemma} \label{deform-star}
    Let $f_1\in A'$ and $f_2 \in B'$. Assume $z\i(f_1^\dag) > s_{z\i(f_1)}$ (resp. $z(f_2^\dag) > s_{z(f_2)}$) or $r_{z\i(f_1)} > r_{f_1}$ (resp. $r_{z(f_2)} > r_{f_2}$). Then there exists a morphism \[\text{$\D: \BA^1 \times \Sigma_z^{f_1, *} \to \Sigma_z^{f_1,*}$ (resp. $\D: \BA^1 \times \Sigma_z^{*, f_2} \to \Sigma_z^{*, f_2}$) }\] such that \begin{itemize}
        \item the restriction morphisms $\D_\t = \D(\t, -): \Sigma_z^{f_1, *} \to \Sigma_z^{f_1, *}$ (resp. $\D_\t = \D(\t, -): \Sigma_z^{*, f_2} \to \Sigma_z^{*, f_2}$) for $\t \in \BA^1$ are isomorphisms;

        \item the morphism $\D_\t$ coincides with the action of $(N_{\a_{f_1}^\vee, r_{f_1}}(\t), 1)$ (resp. $(1, N_{\a_{f_2}^\vee, r_{f_2}}(\t))$) or $(1, N_{\a_{z\i(f_1)}^\vee, r_{z\i(f_1)}}(\t))$ (resp. $(N_{\a_{z(f_2)}^\vee, r_{z(f_2)}}(\t), 1)$) on $\Sigma_z^{f_1, *}$ (resp. $\Sigma_z^{*, f_2}$) for $\t \in \BF_{q^N}$.
    \end{itemize}

    In particular, $H_c^*(\Sigma_z^{f_1, *}, \ov\BQ_\ell)[\phi \otimes \phi\i] = 0$ (resp. $H_c^*(\Sigma_z^{*, f_2}, \ov\BQ_\ell)[\phi \otimes \phi\i] = 0$).
\end{lemma}

%\begin{proposition} \label{r-function}Let $f_1\in A'$ and $f_2 \in B'$. (1) if $f_1^\dag(\bx) \ge f_2^\dag(\bx)$ (resp. $f_1^\dag(\bx) \le f_2^\dag(\bx)$), then $H_c^*(\Sigma_z^{f_1, f_2}, \ov\BQ_\ell)[\phi \otimes \phi\i] = 0$ unless $r_{f_1} \ge r_{z\i(f_1)}$ (resp. $r_{f_2} \ge r_{z(f_2)}$);(2) $H_c^*(\Sigma_z^{f_1, *}, \ov\BQ_\ell)[\phi \otimes \phi\i] = 0$ (resp. $H_c^*(\Sigma_z^{*, f_2}, \ov\BQ_\ell)[\phi \otimes \phi\i] = 0$) unless $r_{f_1} \ge r_{z\i(f_1)}$ (resp. $r_{f_2} \ge r_{z(f_2)}$).\end{proposition}
%\begin{proof}We only show the first statement, then second one follows in a similar way by using Lemma \ref{deform} instead of Lemma \ref{deform-star}.Assume $f_1^\dag(\bx) \ge f_2^\dag(\bx)$ (resp. $f_2^\dag(\bx) \ge f_1^\dag(\bx)$) and $r_{z\i(f_1)} > r_{f_1}$ (resp. $r_{z(f_2)} > r_{f_2}$). As $r_{z\i(f_1)} > r_{f_1} \ge f_1^\dag(\bx) \ge f_2^\dag(\bx)$ (resp. $r_{z(f_2)} > r_{f_2} \ge f_2^\dag(\bx) \ge f_1^\dag(\bx)$), we have $r_{z\i(f_1)} > f_2^\dag(\bx)$ (resp. $r_{z(f_2)} > f_1^\dag(\bx)$). Applying Lemma \ref{deform}, there exists a family of automorphisms $\D_\t$ of $\Sigma_z^{f_1, f_2}$ for $\t \in \BA^1$. Moreover, $\D_\t$ which coincides with that action of $(N_{\a_{z\i(f_1)}, r_{z\i(f_1)}}(\t), 1)$ (resp. $(1, N_{\a_{z(f_2)}, r_{z(f_2)}}(\t))$) when $\t \in \BF_{q^N}$. By definition, $\phi$ is non-trivial over $(N_{\a_{z\i(f_1)}, r_{z\i(f_1)}}(\BF_{q^N})$ (resp. $N_{\a_{z(f_2)}, r_{z(f_2)}}(\BF_{q^N}))$) by definition. Thus we deduce that $H_c^i(\Sigma_z^{f_1, f_2}, \ov\BQ_\ell)[\phi \otimes \phi\i] = 0$ as desired. \end{proof}

\subsection{End of the proof}

Now we are ready to prove  Proposition \ref{red}.
\begin{proof}
    We only prove that $H_c^*(\Sigma_z^{f_1, f_2}, \ov\BQ_\ell)[\phi \otimes \phi\i] = 0$. The other two vanishing statements $H_c^*(\Sigma_z^{f_1, *}, \ov\BQ_\ell)[\phi \otimes \phi\i] = H_c^*(\Sigma_z^{*, f_2}, \ov\BQ_\ell)[\phi \otimes \phi\i] =0$ are proved similarly. By Lemma \ref{value} and the definitions of $A$ and $B$, we have $z(f_2^\dag)(\bx) \ge f_2^\dag(\bx)$ and $z\i(f_1^\dag)(\bx) \ge f_1^\dag(\bx)$.

    Case (1): either $f_1^\dag(\bx) < f_2^\dag(\bx)$ or $f_1^\dag(\bx) = f_2^\dag(\bx)$ and $z(f_2^\dag)(\bx) > f_2^\dag(\bx)$. Hence $z(f_2^\dag)(\bx) > f_1^\dag(\bx)$ and $z(f_2^\dag) > f_1^\dag$. If $r_{z(f_2)} > r_{f_2}$, then $r_{z(f_2)} > f_1^\dag$ (since $r_{z(f_2)} > r_{f_2} \ge f_2^\dag(\bx) \ge f_1^\dag(\bx)$), and the statement follows by Lemma \ref{deform}. So we can assume that $r_{f_2} \ge r_{z(f_2)}$. By Lemma \ref{gedag} we have $z(f_2^\dag) > s_{z(f_2)}$. The statement also follows by Lemma \ref{deform}.

    Case (2): either $f_1^\dag(\bx) > f_2^\dag(\bx)$ or $f_1^\dag(\bx) = f_2^\dag(\bx)$ and $z\i(f_1^\dag)(\bx) > f_1^\dag(\bx)$. Hence $z\i(f_1^\dag)(\bx) > f_2^\dag(\bx)$ and $z\i(f_1^\dag) > f_2^\dag$. If $r_{z\i(f_1)} > r_{f_1}$, then $r_{z\i(f_1)} > f_2^\dag$ (since $r_{z\i(f_1)} > r_{f_1} \ge f_1^\dag(\bx) \ge f_2^\dag(\bx)$), and the statement follows by Lemma \ref{deform}. So we can assume that $r_{f_1} \ge r_{z\i(f_1)}$. By Lemma \ref{gedag} we have $z\i(f_1^\dag) > s_{z\i(f_1)}$. The statement also follows by Lemma \ref{deform}.

    Case (3): $z\i(f_1^\dag)(\bx) = f_1^\dag(\bx) = f_2^\dag(\bx) = z(f_2^\dag)(\bx)$. By Lemma \ref{value}, we have $z\i(f_1)(\bx) = f_1(\bx)$, which means  $\a_{f_1} > 0$ by the definition of $A$. Similarly we have $z(\a_{f_2}) < 0$. Therefore, $\a_{z(f_2^\dag)} > 0 > \a_{f_1^\dag}$ and hence $z(f_2^\dag) > f_1^\dag$. If $r_{f_2} < r_{z(f_2)}$, then $r_{z(f_2)} > f_1^\dag$ (since $r_{z(f_2)} > r_{f_2} \ge f_2^\dag(\bx) = f_1^\dag(\bx)$) and the statement follows from Lemma \ref{deform}. So we can assume that $r_{f_2} \ge r_{z(f_2)}$. By Lemma \ref{gedag} we have $z(f_2^\dag) > s_{z(f_2)}$. The statement follows again by Lemma \ref{deform}.
\end{proof}

\section{The dimension formula} \label{sec:dim-inter}
Let the notation be as in previous sections. In this section, we deduce a dimension formula for the intertwining space $I_w(\CR_{T_r}^{G_r}(\phi))$ for $w \in G^F$. We recall the triviality condition \[\tag{$\dag$} \text{$\phi$ is trivial over $\CT_\phi(\CO_k)$}.\]

\subsection{The set $\tSigma_z^0$}
Consider the following two subsets.
\begin{align*}
    A'' &=\{f \in \tPhi_{\re, r'}^{>0}; \text{ $f > f^\dag$, either $z\i(f)(\bx) < f(\bx)$ or $z\i(f)(\bx) = f(\bx)$ and $\a_f > 0$}\}; \\ B'' &=\{f \in \tPhi_{\re, r}^{>0}; \text{ $f > f^\dag$, either $z(f)(\bx) < f(\bx)$ or $z(f)(\bx) = f(\bx)$ and $\a_{z(f)} < 0$}\};
\end{align*} 

\begin{lemma} \label{I-group-1}
    The subsets $G_{r'}^{A''} \subseteq \CI_{\phi, r'}^\dag$ and $G_r^{B''} \subseteq \CI_{\phi, r}^\dag$ are subgroups. 
\end{lemma}
\begin{proof}
    Let $\Psi$ be the set of roots $\a \in \Phi$ such that either $\a(z(\bx) - \bx) < 0$ or $\a(z(\bx) - \bx) = 0$ and $\a > 0$. By definition we have \[\tag{a} \text{ $\a + \b \in \Psi$ for any $\a, \b \in \Psi$ such that $\a + \b \in \Phi$}.\] By (a) the product $V := \prod_{\a \in \Psi} G^\a \subseteq G$ is a unipotent subgroup. By Lemma \ref{value} and Proposition \ref{I-group}, we deduce that $G_{r'}^{A''} = \CI_{\phi, r'}^\dag \cap V_{r'}$ is a subgroup of $G_{r'}$. The statement for $G_r^{B''}$ follows in a similar way.    
\end{proof}

We define \[\tSigma_z^0 = \{(x_1, x_2, v_1, v_2, \dz) \in \CI_{\phi, r'}^\dag \times \CI_{\phi, r}^\dag \times G_{r'}^{A''} \times G_r^{B''} \times zT_r; F(x_1 v_1 \dz v_2\i x_2\i) = v_1 \dz v_2\i \},\] on which $T_{r'}^F \times T_r^F$ acts by \[(t_1, t_2): (x_1, x_2, v_1, v_2, \dz) \mapsto (t_1 x_1 t_1\i, t_2 x_2 t_2\i, t_1 v_1 t_1\i, t_2 v_2 t_2\i, t_1 \dz t_2\i).\]

\begin{lemma} \label{fibration}
    There is a $T_{r'}^F \times T_r^F$-equivariant smooth morphism with affine space fibers  \[\th: \tSigma_z^0 \to \Sigma_z^0,\quad\ (x_1, x_2, v_1, v_2, \dz) \mapsto (x_1, x_2, v_1 \dz v_2\i).\] In particular, $H_c^*(\Sigma_z^0, \ov\BQ_\ell)[\phi \otimes \phi\i] = H_c^*(\tSigma_z^0, \ov\BQ_\ell)[\phi \otimes \phi\i]$.
\end{lemma}
\begin{proof}
    Note that $B'' = B \sm B'$. By Lemma \ref{I-group-1}, $G_{r'}^{A''} \subseteq \CI_{\phi, r'}^\dag$ is a subgroup. In particular, $G_{r'}^{A''} = G_{r'}^{A'' \cap A} G_{r'}^{A'' \sm A}$. Note that $A'' \cap A = A \sm A'$ and $A'' \sm A = \{f \in A''; z\i(f)(\bx) > r\}$. Hence for $v_2 \in G_r^{B''} = G_r^{B \sm B'}$ and $v_1 = v_1' v_1''$ with $v_1' \in G_{r'}^{A'' \cap A}$ and $v_1'' \in G_{r'}^{A'' \sm A}$ we have \[v_1 \dz v_2\i = v_1' \dz v_2\i \in \CI z \CI / G_{\bx, r+}.\] Thus $\th$ is well-defined and surjective. It then follows from Lemma \ref{coordinate} that $\th$ is a trivial fibration with fibers isomorphic to the affine space $G_{r'}^{A'' \sm A}$.
    
    %Note that the morphism $\th$ is well-defined, and is the natural pull back of the smooth homogenous map \[G_{r'}^{A''} \times T_r \times G_r^{B''} \to G_{r'}^{A''} z T_r G_r^{B''} \subseteq \CI z \CI / G_{\bx, r+},\quad (v_1, \t, v_2) \mapsto v_1 z \t v_2\i,\] whose fibers are isomorphic to $G_{r'}^E$ with $E = \{f \in A''; z\i(f)(\bx) > r\}$.   
\end{proof}

\subsection{The cohomology of $\tSigma_z^0$}
By the change of variables $x_1 \mapsto x_1 v_1$ and $x_2 \mapsto x_2 v_2$, we have \[\tSigma_z^0 = \{(x_1, x_2, v_1, v_2, \dz) \in \CI_{\phi, r'}^\dag \times \CI_{\phi, r}^\dag \times G_{r'}^{A''} \times G_r^{B''} \times zT_r; F(x_1\dz x_2\i) = v_1 \dz v_2\i \}.\]

Consider the group \[H_{r'} = \{(t_1, t_2) \in T_{r'} \times T_{r}; L(t_1) \dz L(t_2)\i = \dz\}.\] Then $H_{r'}$ acts on $\tSigma_z^0$ by \[(t_1, t_2): (x_1, x_2, v_1, v_2, \dz) \mapsto (t_1 x_1 t_1\i, t_2 x_2 t_2\i, F(t_1) v_1 F(t_1)\i, F(t_2) v_2 F(t_2)\i, F(t_1) \dz F(t_2)\i).\] Note that the actions of $H_{r'}$ and $T_{r'}^F \times T_r^F$ on $\tSigma_z^0$ commute with each other.

\begin{lemma} \label{fix}
     Assume $(\dag)$ holds. We have $H_c^*(\tSigma_z^0, \ov\BQ_\ell)[\phi \otimes \phi\i] \cong \ov\BQ_\ell$ if $F(z) = z \in \tW$ and ${}^z \phi|_{\CT(\CO_k)} = \phi|_{\CT(\CO_k)}$, and $H_c^*(\tSigma_z^0, \ov\BQ_\ell)[\phi \otimes \phi\i] = 0$ otherwise.
\end{lemma}
\begin{proof}
    Let $H_{r', \red}^\circ$ be the reductive part of the identity component $H_{r'}^\circ$ of $H_{r'}$. Then the natural projection maps $H_{r', \red}^\circ \to T_{r', \red}$ and $H_{r', \red}^\circ \to T_{r, \red}$ given by $(t_1, t_2) \mapsto t_1$ and $(t_1, t_2) \mapsto t_2$ respectively are surjective. Hence we have \[(\tSigma_z^0)^{H_{r', \red}^\circ} = \{(x_1, x_2, \dz) \in \CT_{\phi, r'} \times \CT_{\phi, r} \times z T_r; F(x_1 \dz x_2\i) = \dz\},\] which is empty unless $z = F(z) \in \tW$.

    Suppose that $F(z) = z \in \tW$. Then  \[H_c^*(\tSigma_z^0, \ov\BQ_\ell)[\phi \otimes \phi\i] = H_c^*((\tSigma_z^0)^{H_{r', \red}^\circ}, \ov\BQ_\ell)[\phi \otimes \phi\i].\] Let $\a \in \Phi$ with $r_\a \ge 1$. Let $M \subseteq T_{r'}^F \times T_r^F$ be the subset consisting of the following elements \[\text{$(N_{z(\a)^\vee, r_\a}(\t), N_{\a^\vee, r_\a}(\t))$, (resp. $(N_{\a^\vee, r_\a}(\t), N_{z\i(\a)^\vee, r_\a}(\t))$) for $\t \in \BF_{q^N}$}.\] As $F(z) = z$, one checks that $M$ acts on $(\tSigma_z^0)^{H_{r', \red}^\circ}$ trivially. If $r_{z(\a)} < r_\a$ (resp. $r_\a > r_{z\i(\a)}$), then $\phi$ is non-trivial on $N_{\a^\vee, r_\a}(\BF_{q^N})$ but trivial on $N_{z(\a)^\vee, r_\a}(\BF_{q^N})$ (resp. $N_{z\i(\a)^\vee, r_\a}(\BF_{q^N})$). Hence $\phi \otimes \phi\i$ is non-trivial over $M$. We deduce that $H_c^*((\tSigma_z^0)^{H_{r', \red}^\circ}, \ov\BQ_\ell)[\phi \otimes \phi\i] = 0$ unless $ r_{z(\a)} \ge r_\a$ and $r_\a \le r_{z\i(\a)}$ for all $\a \in \Phi$ with $r_\a \ge 1$, that is, $r_{z(\b)} = r_{\b}$ for all $\b \in \Phi$. In particular, ${}^z \CT_{\phi, r} = \CT_{\phi, r}$.

    Suppose ${}^z \CT_{\phi, r} = \CT_{\phi, r}$. As $\CT_{\phi, r}$ is an $F$-stable connected unipotent group, we have a universal homeomorphism\[(\tSigma_z^0)^{H_{r', \red}^\circ} \cong \CT_{\phi, r'} \times (z T_r^F \CT_{\phi, r}).\] Hence \begin{align*} H_c^*((\tSigma_z^0)^{H_{r', \red}^\circ}, \ov\BQ_\ell) &\cong H_c^*(z T_r^F \CT_{\phi, r}, \ov\BQ_\ell) \cong H_c^*((z T_r^F \CT_{\phi, r}) / \CT_{\phi, r}, \ov\BQ_\ell) \\ &\cong H_c^*((z T_r^F / \CT_{\phi, r}^F), \ov\BQ_\ell) \cong \bigoplus_\l {}^z\l \otimes \l\i,\end{align*} where $\l$ ranges over characters of $T_r^F / \CT_{\phi, r}^F$. Hence the statement follows. 
\end{proof}

\begin{theorem} \label{intertwine}
    Assume $(\dag)$ holds. For $w \in G^F$ we have \[\dim I_w(\CR_{T_r}^{G_r}(\phi)) = |\{z \in (W_\bx z_w W_\bx)^F; {}^z \phi |_{\CT(\CO_k)} = \phi|_{\CT(\CO_k)}\}|.\] In particular, $\pm \CR_{T_r}^{G_r}(\phi)$ is irreducible if and only if the stabilizer of $\phi|_{\CT(\CO_k)}$ in $W_\bx^F$ is trivial.
\end{theorem}
\begin{proof}
    It follows from Proposition \ref{geo}, Proposition \ref{red}, Lemma \ref{fibration} and Lemma \ref{fix}. 
\end{proof}

\subsection{Construction of irreducible supercuspidals} \label{subsec:c-ind}
Assume further that $T$ is elliptic. In particular, $T(k) = Z(k) \CT(\CO_k)$. Here $Z$ is the center of $G$. 

As $\phi$ is a character of $T^F$, we can extend each $G_{\bx, 0}^F$-module $H_c^i(X_r, \ov\BQ_\ell)[\phi]$ to a $Z^F G_{\bx, 0}^F$-module so that $Z^F$ acts via $\phi$. In this way, we view $\CR_{T_r}^{G_r}(\phi)$ as a (virtual) $Z^F G_{\bx, 0}^F$-module. 

Suppose that $\pm\CR_{T_r}^{G_r}(\phi)$ is a genuine representation of $G_{\bx, 0}^F$. By taking the compact induction, we construct a representation of $G(k) = G^F$ by  \[\pi_{T, \phi} = \text{c-} \ind_{Z^F G_{\bx, 0}^F}^{G^F} \pm \CR_{T_r}^{G_r}(\phi).\]

Let $[\bx]$ be the image of $\bx$ in the reduced Bruhat-Tits building of $G$ over $k$. We denote by $\tW_{[\bx]} = \{z \in \tW; z([\bx]) = [\bx]\}$ the stabilizer of $[\bx]$ in $\tW$. Let $\pr: \tW = X_*(T) \rtimes W \to W$ be the natural projection.

\begin{corollary} \label{irr-sup}
    Assume $T$ is elliptic. Suppose that $\pm\CR_{T_r}^{G_r}(\phi)$ is an irreducible $G_{\bx, 0}^F$-module and the stabilizer of $\phi|_{\CT(\CO_k)}$ in $\pr(\tW_{[\bx]}^F)$ is trivial. Then $\pi_{T, \phi}$ is an irreducible supercuspidal representation of $G^F$.   
\end{corollary}
\begin{proof}
    Let $w \in G^F$ such that \[\dim I_w(\CR_{T_r}^{G_r}(\phi)) \neq 0.\] By Theorem \ref{intertwine} there exists $z \in (W_\bx z_w W_\bx)^F$ such that ${}^z \phi|_{\CT(\CO_k)} = \phi|_{\CT(\CO_k)}$. In particular, $F(z(\bx)) = F(z)(\bx) = z(\bx)$. Hence the vector $z(\bx) - \bx \in X_*(T) \otimes \BR$ is fixed by $F$. As $T$ is elliptic, we have $z(\bx) - \bx \in X_*(Z) \otimes \BR$ and hence $z \in \tW_{[\bx]}^F$. By the assumption on $\phi$, $\pr(z)$ is trivial, that is, $z \in X_*(Z)$, which implies that $w \in Z^F G_{\bx, 0}^F$. So $\pi_{T, \phi}$ is irreducible and the proof is finished.
\end{proof}

\section{Vanishing of $V_r^F$-invariants} \label{sec:sum}

In this section, we show the compact inductions $\pi_{T, \phi}$ in \S \ref{subsec:c-ind} (with $T$ elliptic) are direct sums of irreducible supercuspidals for a large class of characters $\phi$.

Let $P \subsetneq G$ be a proper parabolic subgroup over $k$. Then there exists a $k$-rational cocharacter $\l: \BG_m \to G$ such that \[P = P_G(\l) := \{g: \lim_{t \to 0} \l(t) g \l(t)\i \text{ exists }\}.\] By Bruhat-Tits theory, there is a maximal torus $S \subseteq G$ over $k$ such that (1) $S$ contains a maximally $k$-split torus of $G$ with $\lambda \in X_*(S)$ and (2) $S$ is maximally $\brk$-split and hence $\brk$-split since so is $G$. Then $P = M V$ where $M$ is the Levi subgroup containing $S$ and $V$ is the unipotent radical of $P$. In particular, there exists $\l = F(\l) \in X_*(S)$ such that the sets of roots for $S$ in $P$ and $M$ are given by $\Phi(P, S) = \{\a \in \Phi(G, S); \a(\l) \ge 0\}$ and $\Phi(M, S) = \{\a \in \Phi(G, S); \a(\l) = 0\}$ respectively. Let $\ov V$ be the unipotent radical of the opposite of $P$. We denote by $\tPhi_{\ov V}$ the set of affine roots for $S$ whose finite parts appear in $\ov V$.

Let $W_{M, r} = N_M(S)_r / S_r$ be the Weyl group of $M_r$, where $N_{M}(S)$ is the normalizer of $S$ in $M$. Let $N_G(S, T) = \{x \in G; {}^x S = T\}$ and $W_{S, T, r} = T_r \backslash N_G(S, T)_r / S_r$.  Then we have a decomposition \[G_r = \bigsqcup_{z \in W_{S, T, r} / W_{M, r}} \CI_{z, r} z P_r,\] where $\CI_{z, r} = \CI_r \cap {}^z \ov V_r$. 

Let $z \in W_{S, T, r} / W_{M, r}$. We set $\tPhi_{z, r} = \{f \in \tPhi_{\re, r}^{>0}; z\i(f) \in \tPhi_{\ov V}\}$. Then \[G_r^{\tPhi_{z, r}} = \CI_{z, r}.\] For $v \in \CI_{z, r}$ we denote by $v_f \in \BA_f = \ov\BF_q$ the coordinate of $v$ at $f \in \tPhi_{z, r}$ via the above identification.

\subsection{The reduction step}
Consider the variety \[\Xi = \{(x, \psi) \in \CI_{\phi, r}^\dag \times G_r / V_r; F(x\psi) = \psi\},\] on which $T_r^F$ acts by $t: (x, \psi) \mapsto (txt\i, t\psi)$.
\begin{lemma} \label{V-geo}
    We have $\CR_{T_r}^{G_r}(\phi)^{V_r^F} \cong H_c^*(\Xi, \ov\BQ_\ell)[\phi\i]$.
\end{lemma}
\begin{proof}
    It follows from Proposition \ref{quotient}.
\end{proof}

There is a decomposition \[\Xi = \bigsqcup_{z \in W_{S, T, r}/W_{M, r}} \Xi_z,\] where \begin{align*}\Xi_z &= \{(x, \psi) \in \Xi; \psi \in \CI_r z P_r / V_r\} \\ &=\{(x, vzm V_r) \in \CI_{\phi, r}^\dag \times \CI_{z, r}zP_r/V_r; F(x v z m V_r) = v z m V_r\} \\ &\cong \{(x, v, zm V_r) \in \CI_{\phi, r}^\dag \times \CI_{z, r} \times zP_r/V_r; F(x v z m V_r) = v z m V_r\}. \end{align*}

Let $\tPhi_{z, r}' = \{f \in \tPhi_{z, r}; f < f^\dag\}$. There is a decomposition \[\Xi_z = \Xi_z^0 \sqcup \bigsqcup_{f \in \tPhi_{z, r}'} \Xi_z^f,\] where $\Xi_z^0$ is defined by the condition that $v \in \CI_{\phi, r}^\dag \cap {}^z \ov V_r$, and $\Xi_z^f$ is defined by the condition that $v_f \neq 0$ and $v_h = 0$ for all $h \in \tPhi_{z, r}'$ such that $h^\dag > f^\dag$.

\begin{lemma} \label{vanish}
    We have $H_c^*(\Xi_z^f, \ov\BQ_\ell)[\phi\i] = 0$ for $f \in \tPhi_{z, r}'$.
\end{lemma}
\begin{proof}
We construct a morphism \[\D: \BA^1 \times \Xi_z^f\rightarrow \Xi_z^f, \quad (\t, x, vzmV_r) \mapsto (x', F(t) vzmV_r),\quad t=N_{\a^\vee_f,r_f}(\t).\] To determine $x'$ we need to show that \[\tag{a} F(t) vzm V_r = F(x' F(t) vzm) V_r.\] By definition, we have \[F(t) vzmV_r = F(t x vzm)V_r.\] Thus (a) is equivalent to \[\tag{b} x'F(t) vzm V_r = t x vzm V_r.\] In particular, (b) follows from the identity \[x' F(t) v = t x v \d\] for a suitable element $\d \in G_r$ such that ${}^{(zm)\i} \d \in V_r$ (or equivalently, ${}^{z\i} \d \in V_r$) and \[x' = {}^t(x v \d v\i L(t)\i) \in \CI_{\phi, r}^\dag.\] 

Let $\d = u_{f^\dag}(v_f\i(F^N(\t) - \t))$. By assumption, we have $z\i(\a_f) \in \Phi_{\ov V}$ and hence ${}^{z\i} \d \in \CG_r^{z\i(f^\dag)} \subseteq V_r$. The rest of the proof proceeds as in the proof of Lemma \ref{deform}.    
\end{proof}

\subsection{The vanishing result}
Using the change of variables $x \mapsto x v\i$ we have \[\Xi_z^0 \cong \{(x, v, zmV_r) \in \CI_{\phi, r}^\dag \times (\CI_{\phi, r}^\dag \cap {}^z \ov V_r) \times zP_r/V_r; F(x z m V_r) = v z m V_r\},\] on which $T_r^F$ acts by $t: (x, v, zmV_r) \mapsto (txt\i, tvt\i, tzmV_r)$.

We set $\mu = z(\l) \in X_*(T)$. We define $N_\chi = N_F^{F^N} \circ \chi \circ \left[-\right]: \BG_m \to G_r$ for $\chi=\l$ or $\mu$, where $\left[-\right]$ denotes the Teichm\"uller lift in mixed characteristic and the natural inclusion in equal characteristic. We abbreviate the image of $\chi([t])$ in $G_r$ by $\chi(t)$.
\begin{lemma} \label{empty}
    There is an action of $\BG_m$ on $\Xi_z^0$ given by \[\t: (x, v, zmV_r) \mapsto ({}^{N_\mu(\t)} x , {}^{F(N_\mu(\t))} v, F(N_\mu(\t)) zm F(N_\l(\t\i)) V_r).\] As a consequence, if $T$ is elliptic, then $(\Xi_z^0)^{\BG_m} = \emptyset$ and $H_c^*(\Xi_z^0, \ov\BQ_\ell)[\phi\i] = 0$.
\end{lemma}
\begin{proof}
   As $F(\l) = \l$ is a central cocharacter of $M$ and $z(\l) = \mu$, we have \begin{align*} &\quad\ L(N_\mu(\t)) zm L(N_\l(\t))\i \\ &= F^N(\mu(\t)) \mu(\t\i) zm F^N(\l(\t)\i) \l(\t) \\ &= \mu(F^N(\t)\t\i) z m \l(F^N(\t)\i\t) \\ &= \mu(F^N(\t)\t\i) z \l(F^N(\t)\i\t) m \\ &= zm. \end{align*} Hence the first statement follows. 
   
   Suppose $(x, v, zmV_r) \in (\Xi_z^0)^{\BG_m}$. Then $zm V_r = F(N_\mu(\t)) zm F(N_\l(\t\i)) V_r$, that is, $F(N_\mu(\t)) = {}^z F(N_\l(\t))$ for $\t \in \BG_m$. More precisely, we have \[\sum_{i=0}^{N-1} q^i F^i(\mu) = \sum_{i=0}^{N-1} q^i \mu,\] which implies that $F(\mu) = \mu$. As $T$ is elliptic,  $\mu$ is a central cocharacter of $G$, which contradicts that $P \subsetneq G$ is a proper parabolic subgroup. Hence $(\Xi_z^0)^{\BG_m} = \emptyset$. As the actions of $\BG_m$ and $T_r^F$ on $\Xi_z^0$ commute with each other, we have $H_c^*(\Xi_z^0, \ov\BQ_\ell)[\phi\i] = 0$ and the proof is finished.    
\end{proof}

\begin{theorem} \label{V-vanish}
    If $T$ is elliptic, then $\CR_{T_r}^{G_r}(\phi)^{V_r^F} = 0$.
\end{theorem}
\begin{proof}
    It follows from Lemma \ref{V-geo}, Lemma \ref{vanish} and Lemma \ref{empty}.
\end{proof}

\subsection{The application}
Let the notation be as in \S\ref{subsec:c-ind}. We recall the following standard result on parahoric subgroups, which is an immediate corollary of the Cartan decomposition.

\begin{proposition} \label{finite-parabolic}
    Let $H \subseteq G_{\bx, 0}^F$ be a normal open compact subgroup. Then for all but finitely many $w \in (Z^F G_{\bx, 0}^F) \backslash G^F / G_{\bx, 0}^F$ there exists a $k$-rational proper parabolic subgroup $P \subsetneq G$ containing $S$ with unipotent radical $V$ such that $V^F \cap G_{\bx, 0} \subseteq w H w\i$. 
\end{proposition}

\begin{corollary} \label{direct-sum}
    Assume $T$ is elliptic and $\pm \CR_{T_r}^{G_r}(\phi)$ is a genuine representation of $Z^F G_{\bx, 0}^F$. Then \[\pi_{T, \phi} = \text{c-} \ind_{Z^F G_{\bx, 0}^F}^{G^F} \pm \CR_{T_r}^{G_r}(\phi)\] is a direct sum of irreducible supercuspidal representations of $G^F$.
\end{corollary}
\begin{proof}
    The strategy is the same as in the proof of \cite[Theorem 6.1]{CI23}. Thanks to Bushnell \cite[Theorem 1]{Bush90}, it suffices to show $\pi = \text{c-} \ind_{Z^F G_{\bx, 0}^F}^{G^F} \pm \CR_{T_r}^{G_r}(\phi)$ is an admissible representation of $G^F$. To this end, it suffices to show that for every normal open compact subgroup $H \subseteq G_{\bx, 0}^F$ the space \begin{align*} \pi^H &= \bigoplus_{g \in (Z^F G_{\bx, 0}^F) \backslash G^F / H} (\pm \CR_{T_r}^{G_r}(\phi))^{Z^F G_{\bx, 0}^F \cap g H g\i} \\ & = \bigoplus_{w \in (Z^F G_{\bx, 0}^F) \backslash G^F / G_{\bx, 0}^F} ((\pm \CR_{T_r}^{G_r}(\phi))^{Z^F G_{\bx, 0}^F \cap w H w\i} )^{\oplus m_w}\end{align*} is finite-dimensional, where $m_w\le [G_{\bx,0}^F:H]$ is the number of $Z^F G_{\bx, 0}^F$-$H$ double cosets contained in $Z^F G_{\bx, 0}^F w G_{\bx, 0}^F$. By Proposition \ref{finite-parabolic}, for all but finitely many $w \in (Z^F G_{\bx, 0}^F) \backslash G^F / G_{\bx, 0}^F$ there exists a $k$-rational proper parabolic subgroup $P \subsetneq G$ containing $S$ with unipotent radical $V$ such that $V^F \cap G_{\bx, 0} \subseteq w H w\i$, which, combined with Theorem \ref{V-vanish}, implies that \[(\pm \CR_{T_r}^{G_r}(\phi))^{Z^F G_{\bx, 0}^F \cap w H w\i} \subseteq (\pm \CR_{T_r}^{G_r}(\phi))^{V_r^F} = \{0\}.\] Thus, $\pi^H$ is a finite direct sum of subspaces in $\pm\CR_{T_r}^{G_r}(\phi)$, and hence is finite dimensional as desired.   
\end{proof}


\begin{thebibliography}{GKP00}

\bibitem{Ad}
J.~Adler, \emph{Refined anisotropic $K$-types and supercuspidal representations}, Pacific J.
Math. 185 (1998), 1--32.


\bibitem{BC24}
R.~Bezrukavnikov and C.~Chan, \emph{Generic character sheaves on parahoric subgroups}, arXiv:2401.07189.

\bibitem{Bush90}
C.~Bushnell, \emph{Induced representations of locally profinite groups}, J. Algebra 134 (1990), 104--114.

\bibitem{BK93}
C.~Bushnell and P.~Kutzko, \emph{The Admissible dual of $\GL(N)$ via compact open
subgroups}, Annals of Mathematics Studies, vol. 129, Princeton University Press, Princeton, NJ, 1993

\bibitem{BT1} 
F. Bruhat et J. Tits: \emph{Groupes r\'{e}ductifs sur un corps local, Chapitre I}, Publ. Math. I.H.E.S. 41 (1972), 5--251.

\bibitem{BT2}
F.~Bruhat et J.~Tits, \emph{Groupes r\'{e}ductifs sur un corps local, Chapitre II}, Publ. Math. I.H.E.S. 60 (1984), 5--184.


\bibitem{BW16}
M.~Boyarchenko and J.~Weinstein, \emph{Maximal varieties and the local Langlands correspondence for $\GL(n)$}, J. Amer. Math. Soc. 29 (2016), 177--236.

\bibitem{Chan20} C.~Chan, \emph{The cohomology of semi-infinite Deligne-Lusztig varieties}, J. Reine Angew. Math. 768 (2020), 93--147.


\bibitem{Chan24}
C.~Chan, \emph{The scalar product formula for parahoric Deligne-Lusztig inductions}, arXiv:2405.00671.

\bibitem{CI21} C.~Chan and A.~Ivanov, \emph{Cohomological representations of parahoric subgroups}, Represent. Theory 25 (2021), 1--26.


%\bibitem{CI21b} C.~Chan and A.~Ivanov, \emph{The Drinfeld stratification for $\GL_n$}, Selecta Math. (N.S.) 27 (2021), 54 pp.


%\bibitem{CI21c} C.~Chan and A.~Ivanov, \emph{Affine Deligne-Lusztig varieties at infinite level}, Math. Ann. 380 (2021), 1801--1890.


\bibitem{CI23}
C.~Chan and A.~Ivanov, \emph{On loop Deligne-Lusztig varieties of Coxeter-type for inner forms of $\GL_n$}, Camb. J. Math. 11 (2023), 441--505.


\bibitem{CO25a}
C.~Chan and M.~Oi, \emph{Geometric L-packets of Howe-unramified toral supercuspidal representations}, J. Eur. Math. Soc. (2023), 1--62.


\bibitem{CO25b}
C.~Chan and M.~Oi, \emph{Green functions for positive-depth Deligne-Lusztig induction}, arXiv:2506.04449.


\bibitem{CS17} Z.~Chen and A.~Stasinski, \emph{The algebraisation of higher Deligne-Lusztig representations}, Selecta Math. (N.S.) 23 (2017), 2907--2926.

%\bibitem{Chen18} Z.~Chen, \emph{Green functions and higher Deligne-Lusztig characters}, Doc. Math. 23 (2018), 2027–2041.

\bibitem{CS23}
Z.~Chen and A.~Stasinski, \emph{The algebraisation of higher level Deligne-Lusztig representations II: odd levels}, arXiv: 2311.05354.


\bibitem{DL76}
P.~Deligne and G.~Lusztig, \emph{Representations of reductive groups over finite fields}, Ann. Math. 103 (1976), 103--161.

%\bibitem{DI} O.~Dudas and A.~Ivanov, \emph{Orthogonality relations for deep level Deligne-Lusztig schemes of Coxeter type}, Forum Math. Sigma 12 (2024), 27 pp.


%\bibitem{F21a} J.~Fintzen, \emph{On the construction of tame supercuspidal representations}, Compos. Math. 157 (2021), 2733--2746.


%\bibitem{F21b} J.~Fintzen, \emph{Types for tame  $p$-adic groups}, Ann. of Math. (2) 193 (2021), 303--346.


%\bibitem{FKS} J.~Fintzen, T.~Kaletha and L.~Spice, \emph{A twisted Yu construction, Harish-Chandra characters, and endoscopy}, Duke Math. J. 172 (12), 2241--2301. 

\bibitem{FS}
J.~Fintzen and D.~Schwein, \emph{Construction of tame supercuspidal representations in arbitrary residue characteristic}, arXiv:2501.18553.


\bibitem{Ger75a} 
P.~G\'{e}rardin, \emph{On cuspidal representations of  p -adic reductive groups}, Bull. Amer. Math. Soc. 81 (1975), 756--758.


\bibitem{Ger75b} 
P.~G\'{e}rardin, \emph{Construction de s\'{e}ries discr\`{e}tes $p$-adiques}, Lecture Notes in Mathematics, Vol. 462. Springer-Verlag, Berlin-New York, 1975.


%\bibitem{Howe}R.~Howe, \emph{Tamely ramified supercuspidal representations of $\GL_n$}, Pacific J. Math. 73 (1977), 437--460.

%\bibitem{HM} J.~Hakim and F.~Murnaghan, \emph{Distinguished tame supercuspidal representations}, Int. Math. Res. Pap. IMRP (2008), 166 pp.


\bibitem{Iva23}
A.~Ivanov, \emph{Arc-descent for the perfect loop functor and $p$-adic Deligne-Lusztig spaces}, J. Reine Angew. Math. 794 (2023), 1--54.


%\bibitem{I23b} A.~Ivanov, \emph{On a decomposition of $p$-adic Coxeter orbits}, \'{E}pijournal G\'{e}om. Alg\'{e}brique 7 (2023), 41 pp.


\bibitem{IN24} A.~Ivanov, S.~Nie, \emph{The cohomology of $p$-adic Deligne-Lusztig schemes of Coxeter type}, J. Inst. Math. Jussieu 24 (2025), 1429--1462.


\bibitem{IN25} 
A.~B.~Ivanov, S.~Nie, \emph{Convex elements and deep level Deligne-Lusztig varieties}, arXiv:2503.13412.


\bibitem{IN26} 
A.~B.~Ivanov, S.~Nie, \emph{Deep level Deligne--Lusztig induction for tamely ramified tori}, arXiv:2601.06929.


\bibitem{INY25} 
A.~B.~Ivanov, S.~Nie and Z.~Yu, \emph{An alternative construction of character sheaves on parahoric subgroups}, arXiv:2509.18442.


\bibitem{INT24} 
A.~Ivanov, S.~Nie and P.~Tan, \emph{Inner products of deep level Deligne--Lusztig representations of Coxeter type}, arXiv:2407.18694.


\bibitem{Kal}
T.~Kaletha, \emph{Regular supercuspidal representations}, J. Amer. Math. Soc. 32 (2019), 1071--1170.

\bibitem{LN26}
B.~Liu and S.~Nie, \emph{Symmetric spaces over a finite ring}, arXiv:2608.17765.

\bibitem{Lu79}
G.~Lusztig, \emph{Some remarks on the supercuspidal representations of $p$-adic semisimple groups}, Proc. Sympos. Pure Math. 1979, 171--175.

\bibitem{Lu04}
G.~Lusztig, \emph{Representations of reductive groups over finite rings}, Represent. Theory 8 (2004), 1--14.

\bibitem{Lu90}
G.~Lusztig, \emph{Symmetric spaces over a finite ring}, The Grothendieck Festschrift. Modern Birkh\"{a}user Classics, vol 88. Birkh\"{a}user, Boston, MA, 2007.


%\bibitem{Sta09} A.~Stasinski, \emph{Unramified representations of reductive groups over finite rings}, Represent. Theory 13 (2009), 636--656.

\bibitem{Kim} J.-L.~Kim, \emph{Supercuspidal representations: an exhaustion theorem}, J. Amer. Math. Soc. 20 (2007), 273--320.


\bibitem{Nie24}
S.~Nie, \emph{Decomposition of higher Deligne-Lusztig representations, preprint}, arXiv:2406.06430.


\bibitem{Spr85}
T.~Springer, \emph{Some results on algebraic groups with involutions}, in Algebraic groups and related topics (Kyoto/Nagoya, 1983), 525--543.

\bibitem{SGA3}
M.~Demazure, A.~Grothendieck. {SGA 3, Sch\'{e}mas en Groupes III}, Lecture Notes in
Mathematics 153. Berlin: Springer, 1970.


\bibitem{SS08}
V.~S\'{e}cherre and S.~Stevens, \emph{Repr\'{e}sentations lisses de $GL_m(D)$. IV. Repr\'{e}sentations
supercuspidales}, J. Inst. Math. Jussieu 7 (2008), 527--574.


\bibitem{Ste08}
S.~Stevens, \emph{The supercuspidal representations of $p$-adic classical groups}, Invent. Math. 172 (2008), 289--352.



\bibitem{Tak26}
Y.~Takaya, \emph{Characteristic-free approaches around Yu's construction}, arXiv:2605.03638.



\bibitem{Yu}
J.-K.~Yu, \emph{Construction of tame supercuspidal representations}, J. Amer. Math. Soc. 14 (2001), 579--622.

\bibitem{HR}
T.~Haines and M.~Rapoport \emph{On parahoric subgroups}, arXiv:0804.3788

\bibitem{Fin21b}
J.~Fintzen, \emph{Types for tame $p$-adic groups}, Ann. of Math. (2) \textbf{193} (2021), 303--346.


\end{thebibliography}
\end{document}